\documentclass{amsart}
\usepackage[T1]{fontenc}
\usepackage[utf8]{inputenc}
\usepackage[english]{babel}
\usepackage{amsmath}
\usepackage{amsfonts}
\usepackage{amssymb}
\usepackage{amsthm}
\usepackage[foot]{amsaddr}
\usepackage{mathtools}
\mathtoolsset{showonlyrefs=true}
\usepackage[colorlinks=true,allcolors=blue]{hyperref}
\usepackage{verbatim}
\usepackage{physics}
\usepackage{esint}
\usepackage{enumitem}
\DeclareMathOperator{\sign}{sign}
\DeclareMathOperator{\supp}{supp}
\allowdisplaybreaks

\calclayout

\newtheorem{thm}{Theorem}[section]
\newtheorem{definition}[thm]{Definition}
\newtheorem{lemma}[thm]{Lemma}
\newtheorem{prop}[thm]{Proposition}
\newtheorem{cor}[thm]{Corollary}
\newtheorem{rem}[thm]{Remark}

\let\epsilon\varepsilon

\begin{document}
\author{Mihály Kovács$^{1,2,3}$}
\thanks{M. Kov\'acs acknowledges the support of the Marsden Fund of the Royal Society of New Zealand through grant no. 18-UOO-143, the Swedish Research Council (VR) through project no. 2017-04274 and the NKFIH through grant no. 131545.}
\author{Mihály A. Vághy$^3$}
\thanks{M. A. Vághy acknowledges the support of the ÚNKP-21-3-I-PPKE-60 National Excellence Program of the Ministry for Innovation and Technology from the source of the National Research, Development and Innovation Fund.}

\address{$^1$Department of Mathematical Sciences, Chalmers University of Technology and University of Gothenburg, SE-41296 Gothenburg, Sweden}
\address{$^2$Department of Differential Equations, Budapest University of Technology and Economids, Műegyetem rkp. 3-9, H-1111 Budapest, Hungary}
\address{$^3$Faculty of Information Technology and Bionics, Pázmány Péter Catholic University, Práter u. 50/a, H-1444 Budapest, Hungary}

\title{Nonlinear semigroups for nonlocal conservation laws}


\begin{abstract}
	We investigate a class of nonlocal conservation laws in several space dimensions, where the continuum average of weighted nonlocal interactions are considered over a finite horizon. We establish well-posedness for a broad class of flux functions and initial data via semigroup theory in Banach spaces and, in particular, via the celebrated Crandall-Liggett Theorem. We also show that the unique mild solution satisfies a Kru\v{z}kov-type nonlocal entropy inequality. Similarly to the local case, we demonstrate an efficient way of proving various desirable qualitative properties of the unique solution.
\end{abstract}

\maketitle

\section{Introduction}
We study the semigroup theory of nonlocal conservation laws of the form
\begin{equation}\label{eq:CP}
	\begin{aligned}
		&\pdv{u}{t}+\int_{\mathbb{R}^n}\sum_{i=1}^k\frac{\phi_i(u,\tau_{\beta_i(h)}u)-\phi_i(\tau_{-\beta_i(h)}u,u)}{\norm{\beta_i(h)}_{\mathbb{R}^n}}\omega_i\qty\big(\beta_i(h))\dd{h}=0,\qquad&&(x,t)\in\mathbb{R}^n\times\mathbb{R}_+;\\
		&u(x,0)=u_0(x),&&x\in\mathbb{R},
	\end{aligned}
\end{equation}
where $\tau_{\pm h}u(x,t)=u(x\pm h,t)$ denote a spatial shift of the conserved quantity $u(x,t)$ and the flux functions $\phi_i:\mathbb{R}\times\mathbb{R}\mapsto\mathbb{R}$ are assumed to be increasing with respect to their first arguments and decreasing with respect to their second arguments, and to have the property $\phi_i(0,0)=0$. The number $1\le k\le n$ denotes the number of subinteractions and the functions $\beta_i:\mathbb{R}^n\mapsto\mathbb{R}^n$ are assumed to be of the form
\begin{equation}
	\beta_i(h)=\sum_{j\in B_i}h_je_j,\qquad h=(h_1,h_2,\dots,h_n),
\end{equation}
where the nonempty, pairwise disjoint sets $B_i\subset\qty{1,2,\dots,n}$ are such that $\bigcup_{i=1}^kB_i=\qty{1,2,\dots,n}$ and $e_j$ denotes the $j$th unit vector in $\mathbb{R}^n$.
The kernel functions $\omega_i\in\mathcal{L}^1(\mathbb{R}^n)\cap\mathcal{L}^{\infty}(\mathbb{R}^n)$ are assumed to be nonnegative with $\norm{\omega_i\qty\big(\beta_i(.))}_{\mathcal{L}^1(\mathbb{R}^n)}=1$. We further assume that the support of the kernel functions are finite and are either
\begin{enumerate}
	\item symmetric around the origin, in which case we further assume that the kernels are even, or
	\item contained in $\mathbb{R}_+^n$ such that the closure contains the origin.
\end{enumerate}
For example, in the context of nonlocal particle flows, the above cases allows us to differentiate between multidirectional and unidirectional flows.

Our main examples for the choice of $k$, $\beta_i$ and $\omega_i$ are as follows.
\begin{enumerate}
	\item If $k=1$ and $\beta_1(h)=h$, then the conservation law \eqref{eq:CP} takes the form
		\begin{equation}\label{eq:example:int}
			\pdv{u}{t}+\int_{\mathbb{R}^n}\frac{\phi_1(u,\tau_hu)-\phi_1(\tau_{-h}u,u)}{\norm{h}_{\mathbb{R}^n}}\omega(h)\dd{h}=0.
		\end{equation}
		This case describes a natural multidirectional generalization of the one-dimensional unidirectional nonlocal pair-interaction model investigated in \cite{Du2017}. In fact, if $n=1$ and $\supp(\omega)\subset\mathbb{R}_+$, the law \eqref{eq:example:int} coincides with the latter.
	\item If $k=n$ and $\beta_i(h)=h_ie_i$ and $\omega_i(h)=\prod_{j=1}^n\tilde{\omega}_j(h_j)$, where the kernel functions $\tilde\omega_j$ have analogous properties to that of $\omega_i$ in $\mathbb{R}$ with $\supp(\tilde{\omega}_j)=(-\delta_j,\delta_j)$ for $\delta_j>0$, then the conservation law \eqref{eq:CP} takes the form
		\begin{equation}
			\pdv{u}{t}+\sum_{i=1}^n\int_{-\delta_i}^{\delta_i}\frac{\phi_i(u,\tau_{h_ie_i}u)-\phi_i(\tau_{-h_ie_i}u,u)}{|h_i|}\tilde{\omega}_i(h_i)\dd{h_i}=0.
		\end{equation}
		Should the underlying model allow such considerations, this case corresponds to interactions that can be unfolded into subinteractions along the individual axes. A clear advantage of this example is the ease of numerical approximation of the integral as described in \cite[Section 3.1]{Du2017}. If $n=1$ and $\supp(\tilde\omega_1)=(0,\delta_1)$ instead, then again, we obtain the one-dimensional unidirectional nonlocal pair-interaction model of \cite{Du2017}, as in the previous special case.
\end{enumerate}
We say that the nonlocal flux functions $\phi_i$ are consistent with the local fluxes $\psi_i$ if $\phi_i(a,a)=\psi_i(a)$ holds for {\color{black}all} $a\in\mathbb{R}$. For consistent flux functions, if in addition, the weighting kernels are smooth with their support approaching zero, both special cases {\color{black}formally} lead to the standard local conservation law
\begin{equation}\label{eq:local}
	\pdv{u}{t}+\sum_{i=1}^n\pdv{\psi_i(u)}{x_i}=0.
\end{equation}

For the formal derivation of \eqref{eq:CP} we utilize the nonlocal vector calculus established in \cite{Du2013,Gunzburger2010}. Let $\boldsymbol{\nu},\boldsymbol{\tilde\nu},\boldsymbol{\alpha}:\mathbb{R}^n\times\mathbb{R}^n\mapsto\mathbb{R}^k$ be vector two-point functions defined by the coordinate functions
\begin{align}
	&\boldsymbol{\nu}_i{\color{black}(u)}(x,y,t)=\phi_i\qty\Big(u(x,t),u\qty\big(x+\beta_i(y-x),t)),\\
	&\boldsymbol{\tilde\nu}_i{\color{black}(u)}(x,y,t)=\phi_i\qty\Big(u\qty\big(x+\beta_i(y-x),t),u(x,t)),\\
	&\boldsymbol{\alpha}_i(x,y)=\frac{\omega_i\qty\big(\beta_i(y-x))}{\norm{\beta_i(y-x)}_{\mathbb{R}^n}}.
\end{align}
Then, the nonlocal point divergence is defined as
\begin{equation}
	\mathcal{D}\qty\big(\boldsymbol{\nu}{\color{black}(u)},\boldsymbol{\tilde\nu}{\color{black}(u)})(x,t)=\int_{\mathbb{R}^n}\qty\big(\boldsymbol{\nu}{\color{black}(u)}(x,y,t)-\boldsymbol{\tilde\nu}{\color{black}(u)}(x,y,t))\cdot\boldsymbol{\alpha}(x,y)\dd{y}
\end{equation}
and repeated changes of variables in the integral gives
\begin{equation}\label{eq:divergence}
	\begin{aligned}
		\mathcal{D}\qty\big(\boldsymbol{\nu}{\color{black}(u)},\boldsymbol{\tilde\nu}{\color{black}(u)})(x,t)&=\int_{\mathbb{R}^n}\sum_{i=1}^k\frac{\phi_i(u,\tau_{\beta_i(h)}u)-\phi_i(\tau_{\beta_i(h)}u,u)}{\norm{\beta_i(h)}_{\mathbb{R}^n}}\omega_i\qty\big(\beta_i(h))\dd{h}\\
		&=\int_{\mathbb{R}^n}\sum_{i=1}^k\frac{\phi_i(u,\tau_{\beta_i(h)}u)-\phi_i(\tau_{-\beta_i(h)}u,u)}{\norm{\beta_i(h)}_{\mathbb{R}^n}}\omega_i\qty\big(\beta_i(h))\dd{h}.
	\end{aligned}
\end{equation}
The theory of abstract balance laws thoroughly discussed in \cite[Section 7]{Du2013} shows that in the absence of external sources a class of nonlocal balance laws are given by
\begin{equation}
	\pdv{u}{t}\mbox{}(x,t)+\mathcal{D}\qty\big(\boldsymbol{\nu}{\color{black}(u)},\boldsymbol{\tilde\nu}{\color{black}(u)})(x,t)=0,
\end{equation}
which, combined with \eqref{eq:divergence}, gives exactly the law \eqref{eq:CP}.

Local conservation and balance laws have been widely used in aerodynamics and Eulerian gas dynamics \cite{LeVeque1991}, pedestrian flows \cite{Schreckenberg2002}, ribosome flows \cite{Raveh2015} and many other fields \cite{Smoller1994} for the past decades. {\color{black}In recent years nonlocality has been introduced in multiple forms. A particular method is considering a nonlocal velocity, often expressed as a spatial convolution. The resulting family of models found many applications, for example for modelling supply chains \cite{Keimer2017,Keimer2018,Shang2011} and traffic flows \cite{Chiarello2018,Goatin2016}. However, these models often lack monotonicity of solutions and sometimes they violate the maximum principle as well. Another approach to spatial nonlocality is considering pointwise interactions weighted by an appropriate integral kernel \cite{Du2012}, which has many applications in the field of peridynamics \cite{Bobaru2016,Katiyar2014,Lyngaas2018}. However, the nonlocal model of \cite{Du2012} failed to preserve monotonicity, hence the authors formalized the nonlocal pair-interaction model in \cite{Du2017}, where they show that it satisfies the desired properties. A final advantage of this model is that it reduces to its local counterpart \eqref{eq:local} as the nonlocal horizon vanishes \cite{Du2017num}, while some other nonlocal models do not have this property \cite{Colombo2019}. Because of these improvements the nonlocal pair-interaction model, similarly to its earlier version, has seen many applications in the field of peridynamics \cite{Alimov2021,Zhao2022}, while similar integral terms can be seen in nonlocal formulations of other problems as well, for example of that of the Allen-Cahn equation \cite{Yao2022}.}

It is well known that the solution of \eqref{eq:CP} (including the local case \eqref{eq:local} as well) may develop spatial discontinuities (shock waves) over time, even if the initial data is smooth. Hence the Cauchy problem must be considered in a weak or generalized sense. However, there might be infinitely many weak solutions of \eqref{eq:CP} for given initial data. This fact lead to the development of additional constraints, such as the entropy condition, selecting the unique, physically relevant weak solution, which in this case is the so-called entropy solution.

The well-posedness of the local conservation law \eqref{eq:local} is a thoroughly investigated problem, heavily influenced by the profound work of Kru\v{z}kov \cite{Kruzkov1970}. Kru\v{z}kov showed uniqueness via a priori estimates and existence using the vanishing viscosity method for bounded and measurable initial data and sufficiently smooth flux functions, thus achieving well-posedness. Existence of entropy solutions can often be proved by the convergence of an appropriate numerical scheme \cite{Crandall1980, Sanders1983} (the technique was first used to prove the existence of weak solutions \cite{Conway1966,Douglis1963}). Another classical framework is nonlinear semigroup theory and, in particular, the celebrated Crandall-Liggett Theorem \cite{Crandall1971}, which was first used to prove well-posedness by Crandall \cite{Crandall1972}. Many combinations of these approaches were developed, a notable example being the approximation of semigroups of contractions \cite{Miyadera1992}.

The well-posedness of the one-dimensional nonlocal Cauchy problem with $\beta_1(h)=h$ was investigated in \cite{Du2017}, where Kru\v{z}kov's method was applied to prove uniqueness and existence was proved by the convergence of an appropriate finite volume scheme. While this approach could be extended for multidimensional non-homogeneous Cauchy problems in some special cases (see our second example above), {\color{black}the method is difficult to apply in the generality of \eqref{eq:CP} if $k<n$. Instead, we will also work with the semigroup framework, which provides an elegant way of handling further problems like inhomogeneous conservation laws \cite{Bothe2003} or error control of finite volume methods \cite{Nochetto2006}. Another particular advantage of semigroup theory is the ability to handle $\mathcal{L}^1(\mathbb{R}^n)$ initial data, while with the methods of \cite{Du2017} one can only show existence and uniqueness for $\mathcal{L}^1(\mathbb{R}^n)\cap\mathcal{L}^{\infty}(\mathbb{R}^n)$ initial data.} The semigroup framework considers generalized solutions of abstract Cauchy problems, often called mild solutions. In general, a mild solution can coincide with a weak solution or an entropy solution or, in some cases, with neither; after proving well-posedness an additional investigation is necessary to determine this.

The main results of the paper are contained in Theorems \ref{thm:A_generates} and \ref{thm:A_go_brr} and Corollary \ref{cor:ICP}. In Theorem \ref{thm:A_generates}, we give appropriate circumstances under which there exists an operator satisfying the assumptions of the Crandall-Liggett Theorem. In Theorem \ref{thm:A_go_brr}, we show that the unique mild solution of \eqref{eq:CP} satisfies a nonlocal Kru\v{z}kov-type entropy inequality and has many other qualitative properties that are desirable from a physical point of view. In Corollary \ref{cor:ICP} we extend the well-posedness to conservation laws under Carathéodory forcing.

The outline of the paper is as follows. In Section \ref{sec:preli}, we introduce notations and the abstract framework. In Section \ref{sec:semi}, we give the necessary definitions and state our main results. Section \ref{sec:proof} contains the proof of the main results. The main steps of the proofs are based on \cite{Crandall1972}, however, there are significant nontrivial differences in the details. The difficulty in carrying out this construction is the absence of flux derivatives rendering the method of integration by parts and thus many simplifying steps inapplicable. Most of these complications can be solved by a formally similar technique obtained via changes of variables in the integrals; the technique is often called integration by parts for difference quotients, see, for example \cite[page 295]{Evans2010}. However, a significant step that cannot be resolved in such manner is the verification of the range condition. Crandall uses a perturbation results to establish this, namely \cite[Theorem 3.2]{Kobayashi1977}, but this approach does not seem to be applicable in the nonlocal setting. Instead, we use a fix-point based approach similar to that of \cite[Chapter 4]{Lions1982} and \cite[Proposition IV.3]{Crandall1983}. Throughout the paper the arguments of the functions $\beta_i$ and $\omega_i$ are omitted unless necessary and $C$ is used as a generic constant that may take on different values at different occurrences.

\section{Preliminaries}\label{sec:preli}
We give a brief introduction of the abstract setting based on \cite{Bothe1999,Crandall1972,Crandall1980rep}.

\subsection{Mild solutions of the abstract Cauchy problem}
Let $X$ be a real Banach space and $A$ be a possibly multivalued operator in $X$ and $J=[0,T]\subset\mathbb{R}$ and $f\in\mathcal{L}^1(J,X)$. Consider the quasi-autonomous Cauchy problem
\begin{equation}
	\begin{aligned}\label{eq:ACP}
		&u'+Au\ni f(t),\qquad& t\in J;\\
		&u(0)=u_0
	\end{aligned}
\end{equation}
for $u_0\in\overline{D(A)}$. We call $u\in\mathcal{C}(J,X)$ a mild solution of \eqref{eq:ACP} if for every $\epsilon>0$ there exists a partition $0=t_0\le t_1\le t_2\le\dots\le t_N$ of $[0,t_N]$ and sequences $\qty{z_1,z_2,\dots,z_N},\qty{f_1,f_2,\dots,f_N}$ in $X$ such that
\begin{equation}
	\begin{aligned}
		&t_i-t_{i-1}<\epsilon,\qquad&i=1,\dots,N\\
		&T-\epsilon<t_N\le T,\\
		&\sum_{i=1}^N\int_{t_{i-1}}^{t_i}\norm{f(s)-f_i}\dd{s}<\epsilon,\\
		&\frac{z_i-z_{i-1}}{t_i-t_{i-1}}+Az_i\ni f_i,\qquad&i=1,\dots,N
	\end{aligned}
\end{equation}
and $\norm{z(t)-u(t)}\le\epsilon$ on $[0,t_N]$, where $z:[0,t_N]\mapsto X$ is defined by
\begin{equation}
	z(t)=z_i\qquad\text{for }t_{i-1}\le t<t_i,~i=1,2,\dots,N.
\end{equation}
The piecewise constant function $z$ is called an $\epsilon$-approximate solution of \eqref{eq:ACP}.

Let $F:J\times\overline{D(A)}\mapsto2^X\backslash\emptyset$. A mild solution of the Cauchy problem
\begin{equation}
	\begin{aligned}
		&u'\in -Au+F(t,u),\qquad t\in J;\\
		&u(0)=u_0
	\end{aligned}
\end{equation}
is a function that is a mild solution of the quasi-autonomous problem
\begin{equation}
	\begin{aligned}
		&u'+Au\ni f(t),\qquad t\in J;\\
		&u(0)=u_0
	\end{aligned}
\end{equation}
with some $f\in\mathcal{L}^1(J,X)$ such that $f(t)\in F\qty\big(t,u(t))$ a.e.

\subsection{Crandall-Liggett Theorem}
Let $X$ be a Banach space and $A$ be a possibly multivalued operator in $X$. The operator $A$ is called accretive if, for any $\lambda>0$ and $x,y\in D(A)$, the inequality
\begin{equation}
	\norm{(x+\lambda u)-(y+\lambda v)}\ge\norm{x-y}
\end{equation}
holds, where $u\in Ax$ and $v\in Ay$. The operator $A$ is called $m$-accretive if it is accretive and the operator $I+\lambda A$ is surjective for $\lambda>0$; that is, we have
\begin{equation}
	R(I+\lambda A)=\bigcup_{x\in D(A)}\bigcup_{v\in Ax}\qty{x+\lambda v}=X.\label{eq:range}
\end{equation}

\begin{thm}[Crandall-Liggett Theorem]
	Let $X$ be a Banach space and $A$ be a possibly multivalued $m$-accretive operator in $X$. Then for $\epsilon>0$ and $u_0\in X$ the problem
	\begin{equation}\label{eq:genCP}
		\begin{aligned}
			&\frac{1}{\epsilon}\qty\big(u_{\epsilon}(t)-u_{\epsilon}(t-\epsilon))+Au_{\epsilon}(t)\ni0,\qquad&&t\ge0;\\
			&u_{\epsilon}(0)=u_0,\qquad&&t<0
		\end{aligned}
	\end{equation}
	has a unique solution $u_{\epsilon}(t)$ on $[0,\infty)$. If $u_0\in\overline{D(A)}$, then $\lim_{\epsilon\rightarrow0}u_{\epsilon}(t)$ converges uniformly to the unique mild solution of \eqref{eq:ACP} in bounded sets and $\qty\big(S(t))_{t\ge0}$ defined by $S(t)u_0=\lim_{\epsilon\rightarrow0}u_{\epsilon}(t)$ is a semigroup of contractions on $\overline{D(A)}$; that is, we have 
	\begin{enumerate}[label=\normalfont(\roman*)]
		\item $S(t):\overline{D(A)}\mapsto\overline{D(A)}$ for $t\ge0$,
		\item $S(t)S(\tau)=S(t+\tau)$ for $t,\tau\ge0$,
		\item $\norm{S(t)v-S(t)w}\le\norm{v-w}$ for $t\ge0$ and $v,w\in D(A)$,
		\item $S(0)=I$,
		\item $S(t)v$ is continuous in the pair $(t,v)$.
	\end{enumerate}
\end{thm}
\section{Statement of new results}\label{sec:semi}
{\color{black} The abstract framework of operator semigroups and, in particular, the fundamental Crandall-Liggett Theorem utilizes the notion of mild solutions. Later we will show that the unique mild solution of the conservation law \eqref{eq:CP} also satisfies a Kru\v{z}kov-type entropy inequality. For the exact formulation of this inequality let us define the function $\eta:\mathbb{R}^n\mapsto\mathbb{R}$ to be an entropy of \eqref{eq:CP} with entropy fluxes $q_i:\mathbb{R}^n\times\mathbb{R}^n\mapsto\mathbb{R}$ given that it is continuously differentiable and the equality
\begin{equation}\label{eq:eta_q}
	\eta'(u)\int_{\mathbb{R}^n}\frac{\phi_i(u,\tau_{\beta_i}u)-\phi_i(\tau_{-\beta_i}u,u)}{\norm{\beta_i}_{\mathbb{R}^n}}\omega_i\dd{h}=\int_{\mathbb{R}^n}\frac{q_i(u,\tau_{\beta_i}u)-q_i(\tau_{-\beta_i}u,u)}{\norm{\beta_i}_{\mathbb{R}^n}}\omega_i\dd{h}
\end{equation}
holds for all $i=1,2,\dots,k$. Then if $u(t,x)$ is a $\mathcal{C}^1$ solution of \eqref{eq:CP} then it also satisfies
\begin{equation}
	\pdv{\eta(u)}{t}+\int_{\mathbb{R}^n}\sum_{i=1}^k\frac{q_i(u,\tau_{\beta_i}u)-q_i(\tau_{-\beta_i}u,u)}{\norm{\beta_i}_{\mathbb{R}^n}}\omega_i\dd{h}=0.
\end{equation}
In the case of an $\eta\in\mathcal{C}^2$ convex entropy standard vanishing viscosity arguments (using integration by parts for difference quotients) show that the inequality
\begin{equation}\label{eq:almost_entropy}
	\int_0^T\int_{\mathbb{R}^n}\eta(u)\pdv{f}{t}\dd{x}\dd{t}+\int_0^T\int_{\mathbb{R}^n}\int_{\mathbb{R}^n}\sum_{i=1}^k\frac{\tau_{\beta_i}f-f}{\norm{\beta_i}_{\mathbb{R}^n}}q_i(u,\tau_{\beta_i})\omega_i\dd{h}\dd{x}\dd{t}\ge0
\end{equation}
holds for any $T>0$, nonnegative $f\in\mathcal{C}_0^{\infty}\qty\big(\mathbb{R}^n\times(0,T))$. Our goal is to utilize classical Kru\v{z}kov-entropies of the form $\eta(u):=\eta(u,c)=|u-c|$, however, in this case, an explicit formula for $q_i$ does not seem to reveal itself. Instead, during the vanishing viscosity derivation we rely on \eqref{eq:eta_q} to arrive at the following definition:
\begin{definition}\label{def:entropy}
	A function $u\in\mathcal{L}^1(\mathbb{R}^n\times(0,T))\cap\mathcal{L}^{\infty}(\mathbb{R}^n\times(0,T))$ is an entropy solution of \eqref{eq:CP} if the inequality
	\begin{equation}
		\begin{aligned}
			&\int_0^T\int_{\mathbb{R}^n}\qty\big|u-c|\pdv{f}{t}\dd{x}\dd{t}\\
			&+\int_0^T\int_{\mathbb{R}^n}\int_{\mathbb{R}^n}\sum_{i=1}^k\frac{\tau_{\beta_i}f\sign_0(\tau_{\beta_i}u-c)-f\sign_0(u-c)}{\norm{\beta_i}_{\mathbb{R}^n}}\qty\big(\phi_i(u,\tau_{\beta_i}u)-\phi_i(c,c))\omega_i\dd{h}\dd{x}\dd{t}\ge0
		\end{aligned}
	\end{equation}
	holds for any $T>0$, nonnegative $f\in\mathcal{C}_0^{\infty}\qty\big(\mathbb{R}^n\times(0,T))$ and $c\in\mathbb{R}$.
\end{definition}

\begin{rem}
	Let the functions $\tilde q_i$ be given by\footnote{As already noted by \cite[Definition 2.2]{Fjordholm2021}, the second line is not identical to the corresponding equation in \cite[p. 2470]{Du2017}, which is assumed to be a misprint. Here we gave a more straightforward formula.}
	\begin{equation}
		\begin{aligned}
			\tilde q_i(a,b,c)&=\phi_i(a\vee c,b\vee c)-\phi_i(a\wedge c,b\wedge c)=\phi_i\qty\big(\max\qty{a,c},\max\qty{b,c})-\phi_i\qty\big(\min\qty{a,c},\min\qty{b,c})\\
			&=\frac{\sign_0(a-c)+\sign_0(b-c)}{2}\qty\big(\phi_i(a,b)-\phi_i(c,c))+\frac{\sign_0(a-c)-\sign_0(b-c)}{2}\qty\big(\phi_i(a,c)-\phi_i(c,b)),
		\end{aligned}
	\end{equation}
	where
	\begin{equation}
		\sign_0(x)=\begin{cases}
			1\quad&x>0,\\
			0\quad&x=0,\\
			-1\quad&x<0.
		\end{cases}
	\end{equation}
	Then the properties of $\phi_i$ after adding and subtracting $\phi_i(c,c)$ imply that
	\begin{equation}
		\sign_0(u-c)\int_{\mathbb{R}^n}\frac{\phi_i(u,\tau_{\beta_i}u)-\phi_i(\tau_{-\beta_i}u,u)}{\norm{\beta_i}_{\mathbb{R}^n}}\omega_i\dd{h}\ge\int_{\mathbb{R}^n}\frac{\tilde q_i(u,\tau_{\beta_i}u,c)-\tilde q_i(\tau_{-\beta_i}u,u,c)}{\norm{\beta_i}_{\mathbb{R}^n}}\omega_i\dd{h},
	\end{equation}
	and thus it seems reasonable to define entropy solutions using $\tilde q_i$ as entropy fluxes corresponding to the entropy $|u-c|$. But, in fact, using the product rule for difference quotients shows that
	\begin{equation}
		\begin{aligned}
			&\int_0^T\int_{\mathbb{R}^n}\int_{\mathbb{R}^n}\sum_{i=1}^k\frac{\tau_{\beta_i}f\sign_0(\tau_{\beta_i}u-c)-f\sign_0(u-c)}{\norm{\beta_i}_{\mathbb{R}^n}}\qty\big(\phi_i(u,\tau_{\beta_i}u)-\phi_i(c,c))\omega_i\dd{h}\dd{x}\dd{t}\\
			&=\int_0^T\int_{\mathbb{R}^n}\int_{\mathbb{R}^n}\sum_{i=1}^k\frac{\tau_{\beta_i}f-f}{\norm{\beta_i}_{\mathbb{R}^n}}\sign_0(\tau_{\beta_i}u-c)\qty\big(\phi_i(u,\tau_{\beta_i}u)-\phi_i(c,c))\omega_i\dd{h}\dd{x}\dd{t}\\
			&+\int_0^T\int_{\mathbb{R}^n}\int_{\mathbb{R}^n}\sum_{i=1}^kf\frac{\sign_0(\tau_{\beta_i}u-c)-\sign_0(u-c)}{\norm{\beta_i}_{\mathbb{R}^n}}\qty\big(\phi_i(u,\tau_{\beta_i}u)-\phi_i(c,c))\omega_i\dd{h}\dd{x}\dd{t}.
		\end{aligned}
	\end{equation}
	Clearly
	\begin{equation}
		\sign_0(\tau_{\beta_i}v-c)\qty\big[\phi_i(v,\tau_{\beta_i}v)-\phi_i(c,c)]\le\phi_i(v\vee c,\tau_{\beta_i}v\vee c)-\phi_i(v\wedge c,\tau_{\beta_i}v\wedge c)=\tilde q_i(v,\tau_{\beta_i}v,c)\label{eq:Bv_ineq1}
	\end{equation}
	and similarly
	\begin{equation}
		-\sign_0(v-c)\qty\big[\phi_i(v,\tau_{\beta_i}v)-\phi_i(c,c)]\le-\tilde q_i(v,\tau_{\beta_i}v,c)
	\end{equation}
	holds, thus
	\begin{equation}
		\qty\big[\sign_0(\tau_{\beta_i}v-c)-\sign_0(v-c)]\qty\big[\phi_i(v,\tau_{\beta_i}v)-\phi_i(c,c)]\le0\label{eq:Bv_ineq2}
	\end{equation}
	and finally
	\begin{equation}
		\begin{aligned} 
			&\int_0^T\int_{\mathbb{R}^n}\int_{\mathbb{R}^n}\sum_{i=1}^k\frac{\tau_{\beta_i}f\sign_0(\tau_{\beta_i}u-c)-f\sign_0(u-c)}{\norm{\beta_i}_{\mathbb{R}^n}}\qty\big(\phi_i(u,\tau_{\beta_i}u)-\phi_i(c,c))\omega_i\dd{h}\dd{x}\dd{t}\\
			&\le\int_0^T\int_{\mathbb{R}^n}\int_{\mathbb{R}^n}\sum_{i=1}^k\frac{\tau_{\beta_i}f-f}{\norm{\beta_i}_{\mathbb{R}^n}}\tilde q_i(u,\tau_{\beta_i}u,c)\omega_i\dd{h}\dd{x}\dd{t};
		\end{aligned}
	\end{equation}
	that is, in some sense, the inequality in Definition \ref{def:entropy} is \emph{more precise} in selecting the physically relevant weak solution than the right-hand side of the above inequality. This precision turns out to be crucial in later steps; the operator defined in Definition \ref{def:A0} does not seem to be accretive with the functions $\tilde q_i$ which is an essential property to derive uniqueness of solutions via the Crandall-Liggett theorem.
\end{rem}}

Throughout the paper difference quotients will be denoted by
\begin{equation}
	D^{y}f=\frac{\tau_yf-f}{\norm{y}_{\mathbb{R}^n}},
\end{equation}
where $y\in\mathbb{R}^n$ and the partial derivative of the $\phi_i$ functions with respect to their first and second argument will be denoted by $\phi_{i,1}'$ and $\phi_{i,2}'$, respectively. For open subsets $\Omega$ of $\mathbb{R}^n$ let $\mathcal{W}^{k,p}(\Omega)$ denote the Sobolev space of functions whose distributional derivatives of order at most $k$ are in $\mathcal{L}^p(\Omega)$. The space $\mathcal{W}_0^{k,p}(\Omega)\subset\mathcal{W}^{k,p}(\Omega)$ denotes the set of functions vanishing at the boundary of $\Omega$ and $\mathcal{W}_{loc}^{k,p}(\Omega)$ denotes the set of locally integrable functions whose restriction to any pre-compact $Q\Subset\Omega$ lies in $\mathcal{W}^{k,p}(Q)$. We will use the standard notation $\mathcal{H}^k(\Omega):=\mathcal{W}^{k,2}(\Omega)$.

We rewrite the nonlocal conservation law \eqref{eq:CP} using the operator
\begin{equation}\label{eq:B}
	Bu=\int_{\mathbb{R}^n}\sum_{i=1}^k\frac{\phi_i(u,\tau_{\beta_i}u)-\phi_i(\tau_{-\beta_i}u,u)}{\norm{\beta_i}_{\mathbb{R}^n}}\omega_i\dd{h}
\end{equation}
as
\begin{equation}\label{eq:CP:B}
	\pdv{u}{t}+Bu=0.
\end{equation}

The following lemma shows that {\color{black}for continuously differentiable fluxes} the operator $B$ maps $\mathcal{W}^{1,p}(\mathbb{R}^n)$ to $\mathcal{L}^p(\mathbb{R}^n)$.

\begin{lemma}\label{lem:DB}
	Let $\phi_i\in\mathcal{C}^1(\mathbb{R}\times\mathbb{R})$ have bounded partial derivatives. Then $v\in\mathcal{W}^{1,p}(\mathbb{R}^n)$ implies $Bv\in\mathcal{L}^p(\mathbb{R}^n)$ for all $1\le p<\infty$. In particular, there is a constant $C=C(p)>0$ such that $\norm{Bv}_{\mathcal{L}^p(\mathbb{R}^n)}\le C\norm{\nabla v}_{\mathcal{L}^p(\mathbb{R}^n)}$ for all $v\in\mathcal{W}^{1,p}(\mathbb{R}^n)$.
\end{lemma}
\begin{proof}
	Let $|\phi_{i,1}'|\le K_{i,1}$ and $|\phi_{i,2}'|\le K_{i,2}$ and $\frac{1}{p}+\frac{1}{q}=1$. Setting $K_i=\max\qty{K_{i,1},K_{i,2}}$ we find that
	\begin{align}\label{eq:Bu_upper}
		\norm{Bv}_{\mathcal{L}^p(\mathbb{R}^n)}^p&=\int_{\mathbb{R}^n}\qty\Bigg|\int_{\mathbb{R}^n}\sum_{i=1}^k\frac{\phi_i(v,\tau_{\beta_i}v)-\phi_i(\tau_{-\beta_i}v,v)}{\norm{\beta_i}_{\mathbb{R}^n}}\omega_i\dd{h}|^p\dd{x}\\
		&\le \int_{\mathbb{R}^n}\qty\Bigg(\int_{\mathbb{R}^n}\sum_{i=1}^kK_i\frac{|v-\tau_{-\beta_i}v|+|\tau_{\beta_i}v-v|}{\norm{\beta_i}_{\mathbb{R}^n}}\omega_i\dd{h})^p\dd{x}\\
		&\le k^{p-1}\sum_{i=1}^kK_i^p\int_{\mathbb{R}^n}\qty\Bigg(\int_{\mathbb{R}^n}\qty\Big(\qty\big|D^{\beta_i}\tau_{-\beta_i}v|+\qty\big|D^{\beta_i}v|)\omega_i\dd{h})^p\dd{x}\\
		&\le k^{p-1}\sum_{i=1}^kK_i^p\norm{\omega_i}_{\mathcal{L}^q(\mathbb{R}^n)}^p\int_{\mathbb{R}^n}\int_{\supp(\omega_i)}\qty\Big(\qty\big|D^{\beta_i}\tau_{-\beta_i}v|+\qty\big|D^{\beta_i}v|)^p\dd{h}\dd{x}\\
		&\le2^{p-1}k^{p-1}\sum_{i=1}^kK_i^p\norm{\omega_i}_{\mathcal{L}^q(\mathbb{R}^n)}^p\int_{\supp(\omega_i)}\norm{\nabla v}_{\mathcal{L}^p(\mathbb{R}^n)}^p\dd{h}= C\norm{\nabla v}_{\mathcal{L}^p(\mathbb{R}^n)}^p,
	\end{align}
	where we used the Lipschitz continuity of $\phi$ in the first inequality, Hölder's inequality in the third inequality and finally Fubini's theorem and \cite[Proposition 9.3(iii)]{Brezis2010} in the fourth inequality.
\end{proof}

The continuity of $B$ is established by our next lemma.

\begin{lemma}\label{lem:Bcont}
	Let the assumptions of Lemma \ref{lem:DB} hold. Then $B$ is continuous from $\mathcal{H}^1(\mathbb{R}^n)$ to $\mathcal{L}^2(\mathbb{R}^n)$.
\end{lemma}
\begin{proof}
	Let $u,v\in\mathcal{H}^1(\mathbb{R}^n)$. Similar estimates as in the proof of Lemma \ref{lem:DB} lead to
	\begin{equation}
		\begin{aligned}
			&\norm{Bu-Bv}_{\mathcal{L}^2(\mathbb{R}^n)}^2=\int_{\mathbb{R}^n}\qty\Bigg(\int_{\mathbb{R}^n}\sum_{i=1}^kD^{\beta_i}\qty\big[\phi_i(\tau_{-\beta_i}u,u)-\phi_i(\tau_{-\beta_i}v,v)]\omega_i\dd{h})^2\dd{x}\\
			&\le C\sum_{i=1}^k\norm{D^{\beta_i}\qty\big[\phi_i(\tau_{-\beta_i}u,u)-\phi_i(\tau_{-\beta_i}v,v)]}_{\mathcal{L}^2(\mathbb{R}^n)}^2\le C\sum_{i=1}^k\norm{\nabla\qty\big[\phi_i(\tau_{-\beta_i}u,u)-\phi_i(\tau_{-\beta_i}v,v)]}_{\mathcal{L}^2(\mathbb{R}^n)}^2\\
			&=C\sum_{i=1}^k\Big\|\phi_{i,1}'(\tau_{-\beta_i}u,u)\nabla\tau_{-\beta_i}u+\phi_{i,2}'(\tau_{-\beta_i}u,u)\nabla u-\phi_{i,1}'(\tau_{-\beta_i}v,v)\nabla\tau_{-\beta_i}v-\phi_{i,2}'(\tau_{-\beta_i}v,v)\nabla v\Big\|_{\mathcal{L}^2(\mathbb{R}^n)}^2.
		\end{aligned}
	\end{equation}
	By introducing mixed terms we find that
	\begin{equation}
		\begin{aligned}\label{eq:Bucont}
			&\norm{Bu-Bv}_{\mathcal{L}^2(\mathbb{R}^n)}^2\le C\sum_{i=1}^k\Big(\big\|\qty\big[\phi_{i,1}'(\tau_{-\beta_i}u,u)-\phi_{i,1}'(\tau_{-\beta_i}v,v)]\nabla\tau_{-\beta_i}u\big\|_{\mathcal{L}^2(\mathbb{R}^n)}\\
			&+\big\|\qty\big[\phi_{i,2}'(\tau_{-\beta_i}u,u)-\phi_{i,2}'(\tau_{-\beta_i}v,v)]\nabla u\big\|_{\mathcal{L}^2(\mathbb{R}^n)}\\
			&+\norm{\phi_{i,1}'(\tau_{-\beta_i}v,v)}_{\mathcal{L}^{\infty}(\mathbb{R}^n)}^2\norm{\nabla\tau_{-\beta_i}(u-v)}_{\mathcal{L}^2(\mathbb{R}^n)}^2+\norm{\phi_{i,2}'(\tau_{-\beta_i}v,v)}_{\mathcal{L}^{\infty}(\mathbb{R}^n)}^2\norm{\nabla(u-v)}_{\mathcal{L}^2(\mathbb{R}^n)}^2\Big).
		\end{aligned}
	\end{equation}
	Let $v$ converge to $u$ in $\mathcal{H}^1(\mathbb{R}^n)$ through a sequence $\qty{u_n}\subset\mathcal{H}^1(\mathbb{R}^n)$ and let $\qty{u_{n_k}}$ be a subsequence of $\qty{u_n}$. Since $u_{n_k}$ also converges to $u$ as $n_k\rightarrow\infty$, there exists a subsequence $\{u_{n_{k_l}}\}$ of $\{u_{n_k}\}$ such that $u_{n_{k_l}}\rightarrow u$ a.e. as $n_{k_l}\rightarrow\infty$. Let $|\phi_{i,1}'|\le K_{i,1}$ and $|\phi_{i,2}'|\le K_{i,2}$ and observe that
	\begin{equation}
		\begin{aligned}
			&\qty\Big|\qty\big[\phi_{i,1}'(\tau_{-\beta_i}u,u)-\phi_{i,1}'(\tau_{-\beta_i}u_{n_{k_l}},u_{n_{k_l}})]\nabla\tau_{-\beta_i}u|\le2K_{i,1}\qty\big|\nabla\tau_{-\beta_i}u|,\\
			&\qty\Big|\qty\big[\phi_{i,2}'(\tau_{-\beta_i}u,u)-\phi_{i,2}'(\tau_{-\beta_i}u_{n_{k_l}},u_{n_{k_l}})]\nabla\tau_{-\beta_i}u|\le2K_{i,2}\qty\big|\nabla u|.
		\end{aligned}
	\end{equation}
	Using the dominated convergence theorem and the continuity of $\phi_i$ we find that the first two terms in \eqref{eq:Bucont} converge to zero as $n_{k_l}\rightarrow\infty$. Similarly, since $\phi_{i,1}'$ and $\phi_{i,2}'$ are bounded and $u_{n_{k_l}}\rightarrow u$ in $\mathcal{H}^1(\mathbb{R}^n)$, the second two terms also converge to zero as $n_{k_l}\rightarrow\infty$.	Since $\qty{u_{n_k}}$ was arbitrary we conclude that each subsequence of the sequence $\norm{Bu-Bu_n}_{\mathcal{L}^2(\mathbb{R})}^2$ has a convergent subsequence with limit zero; that is, the sequence itself converges to zero and the proof is complete.
\end{proof}

{\color{black}
\begin{rem}
	In \cite{Fjordholm2021} the authors consider the case (in one dimension) when $\int_{\mathbb{R}^n}\frac{\omega_i(\beta_i)}{\norm{\beta_i}_{\mathbb{R}^n}}<\infty$. In this case the above calculations can be modified to show that $B:\mathcal{L}^1(\mathbb{R}^n)\mapsto\mathcal{L}^1(\mathbb{R}^n)$ is Lipschitz continuous. Hence, standard contraction mapping principle shows existence and uniqueness without entropy conditions. However, in this special case the kernels $\omega_i$ assign small weight to close interactions and more weight as the interaction distance increases. As such, the model's applicability to physically relevant problems is reduced.
\end{rem}
}

We will consider $X=\mathcal{L}^1(\mathbb{R}^n)$ and proceed by verifying the hypotheses of the Crandall-Liggett Theorem for an appropriate operator $A$ in $\mathcal{L}^1(\mathbb{R}^n)$ that is, in some sense, the generalization of the $B$ of $\eqref{eq:CP:B}$. The operator $A$ will be the closure of the operator $A_0$ defined as follows.

\begin{definition}\label{def:A0}
	Let $A_0$ be the operator in $\mathcal{L}^1(\mathbb{R}^n)$ defined by: $v\in D(A_0)$ and $w\in A_0v$ if
	\begin{enumerate}[label=\normalfont(\roman*),ref=\ref{def:A0}\normalfont(\roman*)]
		\item $v,w\in\mathcal{L}^1(\mathbb{R}^n)$,
		\item \label{A0:phi_L1} $\phi_i(v,\tau_{\beta_i(h)}v)\in\mathcal{L}^1(\mathbb{R}^n)$ for $h\in\supp(\omega_i)$ and $i=1,2,\dots,k$,
		\item \label{A0:entropy} the inequality
			\begin{equation}
				\int_{\mathbb{R}^n}\sign_0(v-c)wf\dd{x}+\int_{\mathbb{R}^n}\int_{\mathbb{R}^n}\sum_{i=1}^kD^{\beta_i}\qty\big[f\sign_0(v-c)]\qty\big(\phi_i(v,\tau_{\beta_i}v)-\phi_i(c,c))\omega_i\dd{h}\dd{x}\ge0\label{eq:A0}
			\end{equation}
			holds for any nonnegative $f\in\mathcal{C}_0^{\infty}(\mathbb{R}^n)$ and $c\in\mathbb{R}$.
	\end{enumerate}
\end{definition}

As we will see later, inequality in Definition \ref{A0:entropy} ensures that if $u\in D(A_0)$ is a solution of the abstract Cauchy problem, then it satisfies the entropy inequality in Definition \ref{def:entropy}. Lemmata \ref{lem:A0_single} and \ref{lem:B_in_A0} show that under appropriate circumstances $A_0$ is single-valued and coincides with $B$, further substantiating our definition.

While the accretivity of $A_0$, and thus the accretivity of its closure $A$, can be established in a straightforward manner using a tool described in \cite[Proposition 2.1]{Crandall1972} (see Proposition \ref{prop:A0_accretive}), the verification of the range condition \eqref{eq:range} is more intricate. In fact, it requires the treatment of the stationary equation
\begin{equation}\label{eq:very_regul}
	u+Bu=g.
\end{equation}

We define the generalized solutions of \eqref{eq:very_regul} in terms of $A$.
\begin{definition}
	Let $g\in\mathcal{L}^1(\mathbb{R}^n)$. Then $u\in\mathcal{L}^1(\mathbb{R}^n)$ is a generalized solution of \eqref{eq:very_regul} if $u\in D(A)$ and $g\in(I+A)u$.
\end{definition}

Our first main result is the following theorem.

\begin{thm}\label{thm:A_generates}
	Let $\phi_i\in\mathcal{W}_{loc}^{1,\infty}(\mathbb{R}\times\mathbb{R})$ and $g\in\mathcal{L}^1(\mathbb{R}^n)$. Then $A$ satisfies the assumptions of the Crandall-Liggett Theorem on $\mathcal{L}^1(\mathbb{R}^n)$ and the unique generalized solution of \eqref{eq:very_regul} is given by $u=(I+A)^{-1}g$.
\end{thm}

Theorem \ref{thm:A_generates} and the Crandall-Liggett Theorem show that a semigroup of contractions is determined by the operator $A$, whose various properties are listed in the next theorem.

\begin{thm}\label{thm:A_go_brr}
	Let the assumptions of Theorem \ref{thm:A_generates} hold and $S$ be the semigroup of contractions on $\overline{D(A)}$ obtained from $A$ via the Crandall-Liggett Theorem on $\mathcal{L}^1(\mathbb{R}^n)$. Let $u,v\in\overline{D(A)}\cap\mathcal{L}^{\infty}(\mathbb{R}^n)$ and $t\ge 0$. Then
	\begin{enumerate}[label=\normalfont(\roman*),ref=\ref{thm:A_go_brr}\normalfont(\roman*)]
		\item {\rm(integrability)} $S(t)v\in\mathcal{L}^p(\mathbb{R}^n)$ for $p\ge1$ with $\norm{S(t)v}_{\mathcal{L}^p(\mathbb{R}^n)}\le\norm{v}_{\mathcal{L}^1(\mathbb{R}^n)}^{\frac{1}{p}}\norm{v}_{\mathcal{L}^{\infty}(\mathbb{R}^n)}^{1-\frac{1}{p}}$,
		\item {\rm(maximum principle)} $-\norm{v^-}_{\mathcal{L}^{\infty}(\mathbb{R}^n)}\le S(t)v\le\norm{v^+}_{\mathcal{L}^{\infty}(\mathbb{R}^n)}$, where $v^-=\max\qty{0,-v}$ and $v^+=\max\qty{0,v}$.
		\item {\rm(monotonicity)} $\norm{(S(t)u-S(t)v)^+}_{\mathcal{L}^1(\mathbb{R}^n)}\le\norm{(u-v)^+}_{\mathcal{L}^1(\mathbb{R}^n)}$,
		\item {\rm(equicontinuity)} if $y\in\mathbb{R}^n$, then
			\begin{equation}
				\int_{\mathbb{R}^n}\qty\big|S(t)v(x+y)-S(t)v(x)|\dd{x}\le\int_{\mathbb{R}^n}\qty\big|v(x+y)-v(x)|\dd{x},
			\end{equation}
		\item {\rm(conservation of mass)} $\int_{\mathbb{R}^n}S(t)v(x)\dd{x}=\int_{\mathbb{R}^n}v(x)\dd{x}$,
		\item \label{enum:entropy} $S(t)v$ satisfies the nonlocal entropy inequality in Definition \ref{def:entropy}.
	\end{enumerate}
\end{thm}

{\color{black}
\begin{rem}
	Note that the properties (iii)-(v) still hold even if $u,v\in\overline{D(A)}$.
\end{rem}
}

\begin{cor}\label{cor:ICP}
	Let	$g:[0,T]\times\overline{D(A)}\mapsto\mathcal{L}^1(\mathbb{R}^n)$ be strongly measurable with respect to $t$ and locally Lipschitz with respect to $u$ such that
	\begin{equation}
		\norm{g(t,u)}_{\mathcal{L}^1(\mathbb{R}^n)}\le c(t)\qty\big(1+\norm{u}_{\mathcal{L}^1(\mathbb{R}^n)})
	\end{equation}
	holds for some $c\in\mathcal{L}^1([0,T])$. Then the Cauchy problem
	\begin{equation}
		\begin{aligned}
			&\pdv{u}{t}+\int_{\mathbb{R}^n}\sum_{i=1}^k\frac{\phi_i(u,\tau_{\beta_i}u)-\phi_i(\tau_{-\beta_i}u,u)}{\norm{\beta_i}_{\mathbb{R}^n}}\omega_i\dd{h}=g(t,u),\qquad&&(x,t)\in\mathbb{R}^n\times(0,T];\\
			&u(x,0)=u_0(x),\qquad&&x\in\mathbb{R}
		\end{aligned}
	\end{equation}
	has a unique mild solution for each $u_0\in\overline{D(A)}$ that depends continuously on $u_0$; that is, the map $u_0(.)\rightarrow u(.,t)$ is continuous in the Banach space $X=\mathcal{L}^1(\mathbb{R}^n)$.
\end{cor}
\begin{proof}
	The statement follows directly from \cite[Theorem 5.2]{Bothe2003}.
\end{proof}
\section{Proofs of the main results}\label{sec:proof}
The following lemma shows that $A_0$ is single-valued for bounded functions.
\begin{lemma}\label{lem:A0_single}
	Let $A_0$ be given by Definition \ref{def:A0} and $v\in D(A_0)\cap\mathcal{L}^{\infty}(\mathbb{R}^n)$. Then $A_0$ is single-valued and the equality
	\begin{equation}
		\int_{\mathbb{R}^n}A_0vf\dd{x}=-\int_{\mathbb{R}^n}\int_{\mathbb{R}^n}\sum_{i=1}^kD^{\beta_i}f\phi_i(v,\tau_{\beta_i}v)\omega_i\dd{h}\dd{x}
	\end{equation}
	holds for any nonnegative $f\in\mathcal{C}_0^{\infty}(\mathbb{R}^n)$.
\end{lemma}
\begin{proof}
	Let $w\in A_0v$. Then by \eqref{eq:A0} for any nonnegative $f\in\mathcal{C}_0^{\infty}(\mathbb{R}^n)$ and $c\in\mathbb{R}$ we have
	\begin{equation}
		\int_{\mathbb{R}^n}wf\dd{x}+\int_{\mathbb{R}^n}\int_{\mathbb{R}^n}\sum_{i=1}^kD^{\beta_i}\qty\big[f\sign_0(v-c)]\qty\big(\phi_i(v,\tau_{\beta_i}v)-\phi_i(c,c))\omega_i\dd{h}\dd{x}\ge0,
	\end{equation}
	thus for $c=\norm{v}_{\mathcal{L}^{\infty}(\mathbb{R}^n)}+1$, we have that
	\begin{equation}
		\int_{\mathbb{R}^n}wf\dd{x}\le-\int_{\mathbb{R}^n}\int_{\mathbb{R}^n}\sum_{i=1}^kD^{\beta_i}f\phi_i(v,\tau_{\beta_i}v)\omega_i\dd{h}\dd{x}.
	\end{equation}
	Similarly, letting $c=-(\norm{v}_{\mathcal{L}^{\infty}(\mathbb{R}^n)}+1)$ yields
	\begin{equation}
		\int_{\mathbb{R}^n}wf\dd{x}\ge-\int_{\mathbb{R}^n}\int_{\mathbb{R}^n}\sum_{i=1}^kD^{\beta_i}f\phi_i(v,\tau_{\beta_i}v)\omega_i\dd{h}\dd{x},
	\end{equation}
	showing that for any $w\in A_0v$, the following equality holds
	\begin{equation}
		\int_{\mathbb{R}^n}wf\dd{x}=-\int_{\mathbb{R}^n}\int_{\mathbb{R}^n}\sum_{i=1}^kD^{\beta_i}f\phi_i(v,\tau_{\beta_i}v)\omega_i\dd{h}\dd{x}.
	\end{equation}
	To show that $A_0v$ is single-valued, suppose that $w_1,w_2\in A_0v$. Then the equality $\int_{\mathbb{R}^n}w_1f\dd{x}=\int_{\mathbb{R}^n}w_2f\dd{x}$ holds for all nonnegative $f\in\mathcal{C}_0^{\infty}(\mathbb{R}^n)$, thus $w_1=w_2$ a.e.
\end{proof}

The following lemma shows that $A_0$ extends $B$ on $\mathcal{C}_0^1(\mathbb{R}^n)$.

\begin{lemma}\label{lem:B_in_A0}
	Let $\phi_i\in\mathcal{C}^1(\mathbb{R}\times\mathbb{R})$ and $A_0$ be given by Definition \ref{def:A0}. Then $\mathcal{C}_0^1(\mathbb{R}^n)\subset D(A_0)$ and for any $v\in\mathcal{C}_0^1(\mathbb{R}^n)$, the equality $A_0v=Bv$ holds.
\end{lemma}
\begin{proof}
	{\color{black} The fact $v\in\mathcal{C}_0^1(\mathbb{R}^n)$ implies that $\phi_i(v,\tau_{\beta_i(h)}v)\in\mathcal{L}^1(\mathbb{R}^n)$ holds for all $h\in\supp(\omega_i)$ and $i=1,2,\dots,k$. Let $f\in\mathcal{C}_0^{\infty}(\mathbb{R}^n)$ be nonnegative and $c\in\mathbb{R}$. Multiply $Bv$ by $\sign_0(v-c)f$ and integrate over $\mathbb{R}^n$ to find that
	\begin{equation}
		\begin{aligned}\label{eq:Bv_main_ineq}
			\int_{\mathbb{R}^n}\sign_0(v-c)fBv\dd{x}=&-\int_{\mathbb{R}^n}\int_{\mathbb{R}^n}\sum_{i=1}^kD^{\beta_i}[f\sign_0(v-c)]\qty\big(\phi_i(v,\tau_{\beta_i}v)-\phi_i(c,c))\omega_i\dd{h}\dd{x};
		\end{aligned}
	\end{equation}
	that is, we have $v\in D(A_0)$ and $Bv\in A_0v$. This, combined with Lemma \ref{lem:A0_single} implies that $A_0v=Bv$ a.e.}
\end{proof}

We will use an efficient tool of Crandall to prove accretivity, characterized by the following definition and the two subsequent lemmata.

\begin{definition}\cite[Definition 2.1]{Crandall1972}
	For $u:\mathbb{R}^n\mapsto\mathbb{R}$ measurable, let
	\begin{equation}
		\sign(u):=\qty\big{v:\mathbb{R}^n\mapsto\mathbb{R}\big||v|\overset{}{\le}1\text{ a.e. and }vu=|u|\text{ a.e.}}.
	\end{equation}
\end{definition}

Note that $\sign_0(u)\in\sign(u)$, thus $\sign(u)$ is always nonempty.

\begin{lemma}\label{lem:sign_acc}\cite[Lemma 2.1]{Crandall1972}
	Let $u,v\in\mathcal{L}^1(\mathbb{R}^n)$ and $\alpha\in\sign(u)$. If $\int_{\mathbb{R}^n}\alpha v\dd{x}\ge0$, then $\norm{u+\lambda v}_{\mathcal{L}^1(\mathbb{R}^n)}\ge\norm{u}_{\mathcal{L}^1(\mathbb{R}^n)}$ holds for $\lambda>0$.
\end{lemma}

\begin{lemma}\label{lem:sign_sequence}\cite[Lemma 2.2]{Crandall1972}
	Let $\qty{\alpha_k},\qty{\beta_k}$ be sequences in $\mathcal{L}^1(\mathbb{R}^n)$ with $\lim\beta_k=\beta$. If $\alpha_k\in\sign(\beta_k)$, then there exists a subsequence $\qty{\alpha_{k_l}}$ and function $\alpha\in\sign(\beta)$ such that $\qty{\alpha_{k_l}}$ converges to $\alpha$ in the weak-star topology on $\mathcal{L}^{\infty}(\mathbb{R}^n)$.
\end{lemma}

\begin{prop}\label{prop:A0_accretive}
	Let $A_0$ be given by Definition \ref{def:A0}. Then $A_0$ is accretive in $\mathcal{L}^1(\mathbb{R}^n)$.
\end{prop}
\begin{proof}
	Let $v\in D(A_0)$ and $w\in A_0v$ and choose $u\in\mathcal{L}^1(\mathbb{R}^n)$ such that Definition \ref{def:A0} (ii) holds. Set $c=u(y)$ and $f(x)=g(x,y)$ in \eqref{eq:A0}, where $g\in\mathcal{C}_0^{\infty}(\mathbb{R}^n\times\mathbb{R}^n)$ is nonnegative. We introduce the notations
	\begin{align}
		&D_1^{\beta_i}g(x,y)=\frac{g(x+\beta_i,y)-g(x,y)}{\norm{\beta_i}_{\mathbb{R}^n}},\\
		&D_2^{\beta_i}g(x,y)=\frac{g(x,y+\beta_i)-g(x,y)}{\norm{\beta_i}_{\mathbb{R}^n}}.
	\end{align}
	{\color{black}
	Integrating over $y$ yields
	\begin{equation}\label{eq:acc:ineq1}
		\begin{aligned}
			&\int_{\mathbb{R}^n}\int_{\mathbb{R}^n}\sign_0\qty\big(v(x)-u(y))w(x)g(x,y)\dd{x}\dd{y}\\
			&+\int_{\mathbb{R}^n}\int_{\mathbb{R}^n}\int_{\mathbb{R}^n}\sum_{i=1}^kD_1^{\beta_i}\qty\big[g(x,y)\sign_0\qty\big(v(x)-u(y))]\qty\Big(\phi_i\qty\big(v(x),v(x+\beta_i))-\phi_i\qty\big(u(y),u(y)))\omega_i\dd{h}\dd{x}\dd{y}\ge0.
		\end{aligned}
	\end{equation}
	Suppose that $u\in D(A_0)$ as well and let $z\in A_0u$. Set $c=v(x)$ and $f(y)=g(x,y)$ in \eqref{eq:A0} and integrate over $x$ to find that
	\begin{equation}\label{eq:acc:ineq2}
		\begin{aligned}
			&\int_{\mathbb{R}^n}\int_{\mathbb{R}^n}\sign_0\qty\big(u(y)-v(x))z(y)g(x,y)\dd{y}\dd{x}\\
			&+\int_{\mathbb{R}^n}\int_{\mathbb{R}^n}\int_{\mathbb{R}^n}\sum_{i=1}^kD_2^{\beta_i}\qty\big[g(x,y)\sign_0\qty\big(u(y)-v(x))]\qty\Big(\phi_i\qty\big(u(y),u(y+\beta_i))-\phi_i\qty\big(v(x),v(x)))\omega_i\dd{h}\dd{y}\dd{x}\ge0.
		\end{aligned}
	\end{equation}
	and adding the inequalities \eqref{eq:acc:ineq1} and \eqref{eq:acc:ineq2} yields
	\begin{equation}
		\begin{aligned}\label{eq:acc_proof:inter}
			&\int_{\mathbb{R}^n}\int_{\mathbb{R}^n}\sign_0\qty\big(v(x)-u(y))\qty\big(w(x)-z(y))g(x,y)\dd{x}\dd{y}\\
			&+\int_{\mathbb{R}^n}\int_{\mathbb{R}^n}\int_{\mathbb{R}^n}\sum_{i=1}^k\Bigg(D_1^{\beta_i}\qty\big[g(x,y)\sign_0\qty\big(v(x)-u(y))]\qty\Big(\phi_i\qty\big(v(x),v(x+\beta_i))-\phi_i\qty\big(u(y),u(y)))\\
			&+D_2^{\beta_i}\qty\big[g(x,y)\sign_0\qty\big(u(y)-v(x))]\qty\Big(\phi_i\qty\big(u(y),u(y+\beta_i))-\phi_i\qty\big(v(x),v(x)))\Bigg)\omega_i\dd{h}\dd{x}\dd{y}\ge0.
		\end{aligned}
	\end{equation}
	Let $\delta\in\mathcal{C}_0^{\infty}(\mathbb{R})$ be nonnegative and even such that $\norm{\delta}_{\mathcal{L}^1(\mathbb{R}^n)}=1$ and
	\begin{align}
		&\lambda(x)=\prod_{i=1}^n\delta(x_i),\label{eq:moll}\\
		&\lambda_{\epsilon}(x)=\frac{1}{\epsilon^n}\lambda\qty\bigg(\frac{x}{\epsilon})
	\end{align}
	for $\epsilon>0$. Let $f\in\mathcal{C}_0^{\infty}(\mathbb{R}^n)$ nonnegative and set
	\begin{equation}
		g(x,y)=f\qty\bigg(\frac{x+y}{2})\lambda_{\epsilon}\qty\bigg(\frac{x-y}{2}).
	\end{equation}
	Setting $2\xi=x+y$, $2\eta=x-y$ in \eqref{eq:acc_proof:inter} yields
	\begin{align}\label{eq:acc_proof:full}
		&\int_{\mathbb{R}^n}\qty\Bigg(\int_{\mathbb{R}^n}\sign_0\qty\big(v(\xi+\eta)-u(\xi-\eta))\qty\big(w(\xi+\eta)-z(\xi-\eta))f(\xi)\dd{\xi})\lambda_{\epsilon}(\eta)\dd{\eta}+\int_{\mathbb{R}^n}\int_{\mathbb{R}^n}\int_{\mathbb{R}^n}\sum_{i=1}^k\frac{\omega_i}{\norm{\beta_i}_{\mathbb{R}^n}}\\
		&\Bigg[\qty\Bigg(f\qty\bigg(\xi+\frac{\beta_i}{2})\lambda_{\epsilon}\qty\bigg(\eta+\frac{\beta_i}{2})\sign_0\qty\big(v(\xi+\eta+\beta_i)-u(\xi-\eta))-f(\xi)\lambda_{\epsilon}(\eta)\sign_0(v(\xi+\eta)-u(\xi-\eta)))\\
		&\times\qty\Big(\phi_i\qty\big(v(\xi+\eta),v(\xi+\eta+\beta_i))-\phi_i\qty\big(u(\xi-\eta),u(\xi-\eta)))\\
		&+\qty\Bigg(f\qty\bigg(\xi+\frac{\beta_i}{2})\lambda_{\epsilon}\qty\bigg(\eta-\frac{\beta_i}{2})\sign_0\qty\big(u(\xi-\eta+\beta_i)-v(\xi+\eta))-f(\xi)\lambda_{\epsilon}(\eta)\sign_0\qty\big(u(\xi-\eta)-v(\xi+\eta)))\\
		&\times\qty\Big(\phi_i\qty\big(u(\xi-\eta),u(\xi-\eta+\beta_i))-\phi_i\qty\big(v(\xi+\eta),v(\xi+\eta)))\Bigg]\dd{h}\dd{\xi}\dd{\eta}\ge0.
	\end{align}
	Denote the integral in parenthesis in the first term of $\eqref{eq:acc_proof:full}$ with $I_f(\eta)$ and the whole second term by $J_f(\eta)$. We want to let $\epsilon\rightarrow0$. Since $I_f$ is bounded and $\norm{\lambda_{\epsilon}}_{\mathcal{L}^1(\mathbb{R}^n)}=1$ we have that
	\begin{equation}
		\liminf_{\epsilon\rightarrow0}\int_{\mathbb{R}^n}I_f(\eta)\lambda_{\epsilon}(\eta)\dd{\eta}\le\limsup_{\norm{\eta}_{\mathbb{R}^n}\rightarrow0}I_f(\eta).
	\end{equation}
	A similar argument after a change of variables shows that
	\begin{equation}
		\begin{aligned}
			&\liminf_{\epsilon\rightarrow0}J_f(\eta)\le\limsup_{\norm{\eta}_{\mathbb{R}^n}\rightarrow0}\int_{\mathbb{R}^n}\int_{\mathbb{R}^n}\sum_{i=1}^k\frac{\omega_i}{\norm{\beta_i}_{\mathbb{R}^n}}\\
			&\times\Bigg[f(\xi+\beta_i)q_i^{(1)}\qty\big(v(\xi+\eta),v(\xi+\eta+\beta_i),u(\xi-\eta+\beta_i))-f(\xi)q_i^{(2)}\qty\big(v(\xi+\eta),v(\xi+\eta+\beta_i),u(\xi-\eta))\\
			&+f(\xi+\beta_i)q_i^{(1)}\qty\big(u(\xi-\eta),u(\xi-\eta+\beta_i),v(\xi+\eta+\beta_i))-f(\xi)q_i^{(2)}\qty\big(u(\xi-\eta),u(\xi-\eta+\beta_i),v(\xi+\eta))\Bigg]\dd{h}\dd{\xi},
		\end{aligned}
	\end{equation}
	where
	\begin{equation}
		\begin{aligned}
			q_i^{(1)}(a,b,c)&=\sign_0(a-c)\qty\big(\phi_i(a,b)-\phi_i(c,c)),\\
			q_i^{(2)}(a,b,c)&=\sign_0(b-c)\qty\big(\phi_i(a,b)-\phi_i(c,c)).
		\end{aligned}
	\end{equation}
	Introducing mixed terms yields
	\begin{align}
		&\liminf_{\epsilon\rightarrow0}J_f(\eta)\le\limsup_{\norm{\eta}_{\mathbb{R}^n}\rightarrow0}\int_{\mathbb{R}^n}\int_{\mathbb{R}^n}\sum_{i=1}^k\frac{\omega_i}{\norm{\beta_i}_{\mathbb{R}^n}}\\
		&\times\Bigg[\qty\big(f(\xi+\beta_i)-f(\xi))q_i^{(1)}\qty\big(v(\xi+\eta),v(\xi+\eta+\beta_i),u(\xi-\eta+\beta_i))\\
		&+\qty\big(f(\xi+\beta_i)-f(\xi))q_i^{(1)}\qty\big(u(\xi-\eta),u(\xi-\eta+\beta_i),v(\xi+\eta+\beta_i))\\
		&f(\xi)\qty\Big(q_i^{(1)}\qty\big(v(\xi+\eta),v(\xi+\eta+\beta_i),u(\xi-\eta))-q_i^{(2)}\qty\big(v(\xi+\eta),v(\xi+\eta+\beta_i),u(\xi-\eta)))\\
		&f(\xi)\qty\Big(q_i^{(1)}\qty\big(u(\xi-\eta),u(\xi-\eta+\beta_i),v(\xi+\eta))-q_i^{(2)}\qty\big(u(\xi-\eta),u(\xi-\eta+\beta_i),v(\xi+\eta)))\Bigg]\dd{h}\dd{\xi}.
	\end{align}
	But then \eqref{eq:Bv_ineq2} shows that the last two terms are nonpositive, thus we conclude that
	\begin{equation}
		\begin{aligned}
			&\liminf_{\epsilon\rightarrow0}J_f(\eta)\le\limsup_{\norm{\eta}_{\mathbb{R}^n}\rightarrow0}\int_{\mathbb{R}^n}\int_{\mathbb{R}^n}\sum_{i=1}^k\frac{\omega_i}{\norm{\beta_i}_{\mathbb{R}^n}}\\
			&\times\Bigg[\qty\big(f(\xi+\beta_i)-f(\xi))q_i^{(1)}\qty\big(v(\xi+\eta),v(\xi+\eta+\beta_i),u(\xi-\eta+\beta_i))\\
			&+\qty\big(f(\xi+\beta_i)-f(\xi))q_i^{(1)}\qty\big(u(\xi-\eta),u(\xi-\eta+\beta_i),v(\xi+\eta+\beta_i))\Bigg]\dd{h}\dd{\xi}=:\limsup_{\norm{\eta}_{\mathbb{R}^n}\rightarrow0}\tilde J_f(\eta).
		\end{aligned}
	\end{equation}

	Choose a sequence $\qty{\eta_k}\subset\mathbb{R}^n$ such that $\norm{\eta_k}_{\mathbb{R}^n}\rightarrow0$ and $\lim_{k\rightarrow\infty}I_f(\eta_k)=\limsup_{\norm{\eta}_{\mathbb{R}^n}}\tilde I_f(\eta)$ and $\lim_{k\rightarrow\infty}\tilde J_f(\eta_k))=\limsup_{\norm{\eta}_{\mathbb{R}^n}\rightarrow0}\tilde J_f(\eta)$ (note that it might be necessary to choose two different sequences for $I_f$ and $\tilde J_f$). Using Lemma \ref{lem:sign_sequence} we assume (passing to subsequences if necessary) that the sequence
	\begin{equation}
		\alpha_k(\xi)=\sign_0\qty\big(v(\xi+\eta_k)-u(\xi-\eta_k))
	\end{equation}
	converges weakly-star in $\mathcal{L}^{\infty}(\mathbb{R}^n)$ to $\alpha\in\sign\qty\big(v(\xi)-u(\xi))$. We similarly assume that the $\sign_0$ sequences appearing in $\tilde J_f(\eta_k)$ converge weakly-star in $\mathcal{L}^{\infty}(\mathbb{R}^n)$ and we denote the limit as
	\begin{equation}
		\lim_{k\rightarrow\infty}\tilde J_f(\eta_k)=:\int_{\mathbb{R}^n}\int_{\mathbb{R}^n}\sum_{i=1}^kD^{\beta_i}f\qty\Big(\gamma_i\qty\big(v(\xi),v(\xi+\beta_i),u(\xi+\beta_i))+\gamma_i\qty\big(u(\xi),u(\xi+\beta_i),v(\xi+\beta_i)))\omega_i\dd{h}\dd{\xi}.
	\end{equation}
	Then
	\begin{equation}\label{eq:lim_gamma}
		\begin{aligned}
			&\lim_{k\rightarrow\infty}\qty\big(I_f(\eta_k)+\tilde J_f(\eta_k))=\int_{\mathbb{R}^n}\alpha(w-z)f\dd{\xi}\\
			&+\int_{\mathbb{R}^n}\int_{\mathbb{R}^n}\sum_{i=1}^kD^{\beta_i}f\qty\Big(\gamma_i\qty\big(v(\xi),v(\xi+\beta_i),u(\xi+\beta_i))+\gamma_i\qty\big(u(\xi),u(\xi+\beta_i),v(\xi+\beta_i)))\omega_i\dd{h}\dd{\xi}\ge0.
		\end{aligned}
	\end{equation}
	Let $\kappa\in\mathcal{C}_0^{\infty}(\mathbb{R})$ be nonnegative such that $\kappa(s)=1$ for $|s|\le1$. Set $f_l(\xi)=\kappa\qty\Big(\frac{\norm{\xi}_{\mathbb{R}^n}}{l})$ and let $l\rightarrow\infty$. Since the difference quotient
	\begin{equation}
		D^{\beta_i}f_l(x)=\int_0^1\nabla f_l(x+\beta_i s)\cdot\frac{\beta_i}{\norm{\beta_i}_{\mathbb{R}^n}}\dd{s}\label{eq:Dbif}
	\end{equation}
	is bounded and is zero for $x\in\mathbb{R}^n$ such that $\norm{x\pm\beta_i}_{\mathbb{R}^n}\le l$, the second integral in \eqref{eq:lim_gamma} converges to zero; that is, we conclude that
	\begin{equation}
		\int_{\mathbb{R}^n}\alpha(w-z)\dd{\xi}\ge0.
	\end{equation}
	Lemma \ref{lem:sign_acc} shows that the inequality
	\begin{equation}
		\norm{v-u+\lambda(w-z)}_{\mathcal{L}^1(\mathbb{R}^n)}\ge\norm{v-u}_{\mathcal{L}^1(\mathbb{R}^n)}
	\end{equation}
	holds for $\lambda>0$. Since $u,v\in D(A_0)$ were arbitrary we conclude that $A_0$ is indeed accretive.}
\end{proof}

{\color{black}
\begin{rem}
	One can observe that in the above proof we did not use the fact that the kernels $\omega_i$ have finite support.
\end{rem}
}

The stationary equation \eqref{eq:very_regul} will be investigated through the regularized equation
\begin{equation}\label{eq:regul}
	u+\lambda Bu-\epsilon\Delta u=g,
\end{equation}
where $\lambda,\epsilon>0$. {\color{black}In \cite[Proposition 2.2]{Crandall1972} the author shows existence of solutions using a special version of the perturbation result \cite[Theorem 3.2]{Kobayashi1977} without further preparations. A key step of the proof is the fact that for $u\in\mathcal{L}^2(\mathbb{R}^n)$, the $\tilde B$ local version of the operator $B$ (see \eqref{eq:local}) has the property $\langle\tilde Bu,u\rangle=0$. However, this is no longer true in the nonlocal case, and thus we instead use a fix-point approach based on \cite[Chapter 4]{Lions1982} and \cite[Proposition IV.3]{Crandall1983}. In order to do so, we first establish some a priori estimates on the solutions.}

\begin{lemma}\label{lem:u_Lp}
	Let $\phi_i\in\mathcal{C}^1(\mathbb{R}\times\mathbb{R})$ have bounded partial derivatives and $u\in\mathcal{H}^2(\mathbb{R}^n)$ satisfy $\eqref{eq:regul}$ for $g\in\mathcal{L}^1(\mathbb{R}^n)\cap\mathcal{L}^{\infty}(\mathbb{R}^n)$. Then $u\in\mathcal{L}^1(\mathbb{R}^n)\cap\mathcal{L}^{\infty}(\mathbb{R}^n)$ and
	\begin{equation}
		\begin{aligned}
			&\norm{u}_{\mathcal{L}^1(\mathbb{R}^n)}\le\norm{g}_{\mathcal{L}^1(\mathbb{R}^n)},\\
			&\norm{u}_{\mathcal{L}^{\infty}(\mathbb{R}^n)}\le\norm{g}_{\mathcal{L}^{\infty}(\mathbb{R}^n)}.
		\end{aligned}
	\end{equation}
\end{lemma}
\begin{proof}
	We treat the case of $\mathcal{L}^1(\mathbb{R}^n)$ first. Define
	\begin{equation}\label{eq:Phi}
		\Phi_l(s)=\begin{cases}-s\quad&\text{if }s\le-\frac{1}{l},\\
			\frac{l}{2}s^2+\frac{1}{2l}\quad&\text{if }|s|\le\frac{1}{l},\\
			s\quad&\text{if }s\ge\frac{1}{l}
		\end{cases}
	\end{equation}
	and let $f\in\mathcal{C}_0^{\infty}(\mathbb{R}^n)$ be such that $0\le f\le1$. Multiplying $\eqref{eq:regul}$ by $\Phi_l'(u)f$ and integrating over $\mathbb{R}^n$ gives
	\begin{equation}\label{eq:regul_estim}
		\int_{\mathbb{R}^n}\qty\big(u\Phi_l'(u)f+\lambda Bu\Phi_l'(u)f-\epsilon\Delta u\Phi_l'(u)f)\dd{x}=\int_{\mathbb{R}^n}g\Phi_l'(u)f\dd{x}\le\norm{g}_{\mathcal{L}^1(\mathbb{R}^n)}.
	\end{equation}
	Since the sequence $\qty\big{u\Phi_l'(u)f}$ is a nonnegative and pointwise non-decreasing sequence with $u\Phi_l'(u)f\rightarrow|u|f$ as $l\rightarrow\infty$, the monotone convergence theorem and the fact that $0\le\Phi_l'f\le1$ implies
	\begin{equation}\label{eq:contr:1}
		\lim_{l\rightarrow\infty}\int_{\mathbb{R}^n}u\Phi_l'(u)f\dd{x}=\int_{\mathbb{R}^n}uf\dd{x}.
	\end{equation}
	Since $\Phi_l'$ is monotone, and $f$ is nonnegative we have that
	\begin{equation}\label{eq:contr:2}
		\begin{aligned}
			\int_{\mathbb{R}^n}\Delta u\Phi_l'(u)f\dd{x}&=-\int_{\mathbb{R}^n}\Phi_l''(u)|\nabla u|^2f\dd{x}-\int_{\mathbb{R}^n}\Phi_l'(u)\nabla u\nabla f\dd{x}\\
			&=-\int_{\mathbb{R}^n}\Phi_l''(u)|\nabla u|^2f\dd{x}+\int_{\mathbb{R}^n}\Phi_l(u)\Delta f\dd{x}\le\int_{\mathbb{R}^n}\Phi_l(u)\Delta f\dd{x}.
		\end{aligned}
	\end{equation}
	By letting $l\rightarrow\infty$ we conclude that
	\begin{equation*}
		-\limsup_{l\rightarrow\infty}\int_{\mathbb{R}^n}\Delta u\Phi_l'(u)f\dd{x}\ge-\int_{\mathbb{R}^n}u\Delta f\dd{x}.
	\end{equation*}
	Finally, the sequence $\qty\big{Bu\Phi_l'(u)f}$ converges pointwise to $Bu\sign_0(u)f$ as $l\rightarrow\infty$ and is dominated by $|Bu|f$. The fact that $|Bu|f$ is integrable follows from Sobolev's embedding of $\mathcal{H}^2$ into $\mathcal{W}^{1,1}$ on the support of $f$ and Lemma \ref{lem:DB}. Thus, using the dominated convergence theorem yields
	\begin{equation}
		\lim_{l\rightarrow\infty}\int_{\mathbb{R}^n}Bu\Phi_l'(u)f\dd{x}=\int_{\mathbb{R}^n}Bu\sign_0(u)f\dd{x}.
	\end{equation}
	Use the integration by parts formula for difference quotients to find that
		\begin{equation}
			\begin{aligned}
				\lim_{l\rightarrow\infty}\int_{\mathbb{R}^n}Bu\Phi_l'(u)f\dd{x}&=-\int_{\mathbb{R}^n}\int_{\mathbb{R}^n}\sum_{i=1}^kD^{\beta_i}\sign_0(u)\tau_{\beta_i}f\phi_i(u,\tau_{\beta_i}u)\omega_i\dd{h}\dd{x}\\
				&-\int_{\mathbb{R}^n}\int_{\mathbb{R}^n}\Phi_l'(u)D^{\beta_i}f\phi_i(u,\tau_{\beta_i}u)\omega_i\dd{h}\dd{x},
			\end{aligned}
	\end{equation}
	and apply inequality \eqref{eq:Bv_ineq2} with $c=0$ to conclude that 
	\begin{equation}
		\lim_{l\rightarrow\infty}\int_{\mathbb{R}^n}Bu\Phi_l'(u)f\dd{x}\ge-\int_{\mathbb{R}^n}\int_{\mathbb{R}^n}\sum_{i=1}^k\sign_0(u)D^{\beta_i}f\phi_i(u,\tau_{\beta_i}u)\omega_i\dd{h}\dd{x}.\label{eq:contr:3}
	\end{equation}
	Substituting \eqref{eq:contr:1}, \eqref{eq:contr:2} and \eqref{eq:contr:3} into $\eqref{eq:regul_estim}$ yields
	\begin{equation}
		\int_{\mathbb{R}^n}uf\dd{x}-\int_{\mathbb{R}^n}\int_{\mathbb{R}^n}\sum_{i=1}^k\sign_0(u)D^{\beta_i}f\phi_i(u,\tau_{\beta_i}u)\omega_i\dd{h}\dd{x}-\epsilon\int_{\mathbb{R}^n}u\Delta f\dd{x}\le\norm{g}_{\mathcal{L}^1(\mathbb{R}^n)}.
	\end{equation}
	Let $\kappa\in\mathcal{C}_0^{\infty}(\mathbb{R})$ nonnegative such that $\kappa(s)=1$ for $|s|\le1$. Set $f_l(\xi)=\kappa\qty\Big(\frac{\norm{\xi}_{\mathbb{R}^n}}{l})$. Since the difference quotient $D^{\beta_i}f_l$ is bounded and is zero for $x\in\mathbb{R}^n$ such that $\norm{x\pm\beta_i}_{\mathbb{R}^n}\le l$  (see \eqref{eq:Dbif}), letting $l\rightarrow\infty$ yields
	\begin{equation}
		\norm{u}_{\mathcal{L}^1(\mathbb{R}^n)}\le\norm{g}_{\mathcal{L}^1(\mathbb{R}^n)}.
	\end{equation}

	For the case of $\mathcal{L}^{\infty}(\mathbb{R}^n)$, let $M\in\mathbb{R}$ be such that $M\ge g^+$ a.e. Subtract $M$ from $\eqref{eq:regul}$, multiply by $\Phi_l'^+(u-M)$ and integrate over $\mathbb{R}^n$ to find that
	\begin{equation}\label{eq:contr2}
		\int_{\mathbb{R}^n}(u-M+\lambda Bu-\epsilon\Delta u)\Phi_l'^+(u-M)\dd{x}=\int_{\mathbb{R}^n}(g-M)\Phi_l'^+(u-M)\dd{x}\le0.
	\end{equation}
	A similar argument as in \eqref{eq:contr:2} gives
	\begin{equation}
		\lim_{l\rightarrow\infty}\int_{\mathbb{R}^n}\Delta u\Phi_l'^+(u-M)\dd{x}\le0,\label{eq:contr2_begin}
	\end{equation}
	as before. Again, integration by parts for difference quotients and the inequality \eqref{eq:Bv_ineq2} with $c=M$ (the reader may want to check that $\sign_0$ and $\sign_0^{\pm}$ are interchangeable in \eqref{eq:Bv_ineq2}) imply that
	\begin{equation}
		\begin{aligned}
			&\lim_{l\rightarrow\infty}\int_{\mathbb{R}^n}Bu\Phi_l'^+(u-M)\dd{x}=-\int_{\mathbb{R}^n}\int_{\mathbb{R}^n}\sum_{i=1}^kD^{\beta_i}\sign_0^+(u-M)\qty\big[\phi_i(u,\tau_{\beta_i}u)-\phi_i(M,M)]\omega_i\dd{h}\dd{x}\ge0.\label{eq:contr2_end}
		\end{aligned}
	\end{equation}
	Substituting \eqref{eq:contr2_begin} and \eqref{eq:contr2_end} into \eqref{eq:contr2} yields
	\begin{equation}
		\int_{\mathbb{R}^n}(u-M)\Phi_l'^+(u-M)\dd{x}\le0,
	\end{equation}
	which implies that $u\le M$ a.e.

	To establish an analogous lower bound, let $M$ be such that $M\le g^-$ a.e. Add $M$ to \eqref{eq:regul}, multiply by $\Phi_l'^-(u+M)$ and integrate over $\mathbb{R}^n$ to conclude that
	\begin{equation}
		\int_{\mathbb{R}^n}(u+M+\lambda Bu-\epsilon\Delta u)\Phi_l'^-(u+M)\dd{x}=\int_{\mathbb{R}^n}(g+M)\Phi_l'(u+M)^-\dd{x}\le0.
	\end{equation}
	Similar estimates as before show that
	\begin{equation}
		\int_{\mathbb{R}^n}(u+M)\Phi_l'^-(u+M)\dd{x}\le0,
	\end{equation}
	which implies that $-M\le u$ a.e. Setting $M=\norm{g}_{\mathcal{L}^{\infty}(\mathbb{R}^n)}$ concludes the proof.
\end{proof}

\begin{rem}\label{rem:infty}
	The proof also shows that the maximum principle holds for equation \eqref{eq:regul}; that is, any solution $u\in\mathcal{H}^2(\mathbb{R}^n)$ of \eqref{eq:regul} satisfies $-\norm{g^-}_{\mathcal{L}^{\infty}(\mathbb{R}^n)}\le u\le\norm{g^+}_{\mathcal{L}^{\infty}(\mathbb{R}^n)}$ a.e.
\end{rem}

Hölder's inequality immediately yields the following result.
\begin{cor}\label{cor:u_Lp}
	Let the assumptions of Lemma \ref{lem:u_Lp} hold and let $g\in\mathcal{L}^1(\mathbb{R}^n)\cap\mathcal{L}^{\infty}(\mathbb{R}^n)$. Then $u\in\mathcal{L}^p(\mathbb{R}^n)$ for $p\ge1$ with $\norm{u}_{\mathcal{L}^p(\mathbb{R}^n)}\le\norm{g}_{\mathcal{L}^1(\mathbb{R}^n)}^{\frac{1}{p}}\norm{g}_{\mathcal{L}^{\infty}(\mathbb{R}^n)}^{1-\frac{1}{p}}$.
\end{cor}

The next result shows the uniqueness of solutions of \eqref{eq:regul} for $g\in\mathcal{L}^1(\mathbb{R}^n)$.

\begin{lemma}\label{lem:regul_L1+}
	Let the assumptions of Lemma \ref{lem:u_Lp} hold and let $u,v\in\mathcal{H}^2(\mathbb{R}^n)$ satisfy
	\begin{align}
		&u+\lambda Bu-\epsilon\Delta u=g_1,\\
		&v+\lambda Bv-\epsilon\Delta v=g_2.
	\end{align}
	If $g_1,g_2\in\mathcal{L}^1(\mathbb{R}^n)$, then
	\begin{equation}
		\norm{(u-v)^+}_{\mathcal{L}^1(\mathbb{R}^n)}\le\norm{(g_1-g_2)^+}_{\mathcal{L}^1(\mathbb{R}^n)}.
	\end{equation}
\end{lemma}
\begin{proof}
	The proof follows the proof of Lemma \ref{lem:u_Lp}. Let $w=u-v$. Then $w$ satisfies
	\begin{equation}\label{eq:uniq:w}
		w+\lambda(Bu-Bv)-\epsilon\Delta w=g_1-g_2.
	\end{equation}
	Let $f\in\mathcal{C}_0^{\infty}(\mathbb{R}^n)$ be such that $0\le f\le 1$. Define $\Psi_l$ by setting $\Psi_l'=\Phi_l'^+$ and $\Psi_l(0)=0$. Multiply \eqref{eq:uniq:w} by $\Psi_l'(w)f$ and integrate over $\mathbb{R}^n$ to find that
	\begin{equation}
		\begin{aligned}\label{eq:uniq_main}
			\int_{\mathbb{R}^n}\qty\big(w+\lambda(Bu-Bv)-\epsilon\Delta w)\Psi_l'(w)f\dd{x}=\int_{\mathbb{R}^n}(g_1-g_2)\Psi_l'(w)f\dd{x}\le\norm{(g_1-g_2)^+}_{\mathcal{L}^1(\mathbb{R}^n)}
		\end{aligned}
	\end{equation}
	holds, since $0\le\Psi_l'f\le1$. The facts that $\Psi_l(w)\in\mathcal{H}^1(\mathbb{R}^n)$ and that both $\Psi_l'',f\ge0$ imply that
	\begin{equation}
		\int_{\mathbb{R}^n}\Delta w\Psi_l'(w)f\dd{x}\le\int_{\mathbb{R}^n}\Psi_l(w)\Delta f\dd{x},
	\end{equation}
	and thus
	\begin{equation}\label{eq:uniq_Laplace}
		-\limsup_{l\rightarrow\infty}\int_{\mathbb{R}^n}\Delta w\Psi_l'(w)f\dd{x}\ge-\int_{\mathbb{R}^n}w^+\Delta f\dd{x}.
	\end{equation}
	as before. Integration by parts for difference quotients yields
	\begin{equation}
		\begin{aligned}
			\int_{\mathbb{R}^n}(Bu-Bv)\Psi_l'(w)f\dd{x}&=-\int_{\mathbb{R}^n}\int_{\mathbb{R}^n}\sum_{i=1}^kD^{\beta_i}\Psi_l'(w)\tau_{\beta_i}f\qty\big[\phi_i(u,\tau_{\beta_i}u)-\phi_i(v,\tau_{\beta_i}v)]\omega_i\dd{h}\dd{x}\\
			&-\int_{\mathbb{R}^n}\int_{\mathbb{R}^n}\sum_{i=1}^k\Psi_l'(w)D^{\beta_i}f\qty\big[\phi_i(u,\tau_{\beta_i}u)-\phi_i(v,\tau_{\beta_i}v)]\omega_i\dd{h}\dd{x}.
		\end{aligned}
	\end{equation}
	Letting $l\rightarrow\infty$ in the first integral and using a similar argument as in \eqref{eq:Bv_ineq2} we find that 
	\begin{equation}
		-\lim_{l\rightarrow\infty}\int_{\mathbb{R}^n}\int_{\mathbb{R}^n}\sum_{i=1}^kD^{\beta_i}\Psi_l'(w)\tau_{\beta_i}f\qty\big[\phi_i(u,\tau_{\beta_i}u)-\phi_i(v,\tau_{\beta_i}v)]\omega_i\dd{h}\dd{x}\ge0,
	\end{equation}
	and thus, by the dominated convergence theorem,
	\begin{equation}\label{eq:uniq_B}
		\begin{aligned}
			&\lim_{l\rightarrow\infty}\int_{\mathbb{R}^n}(Bu-Bv)\Psi_l'(w)f\dd{x}\ge-\int_{\mathbb{R}^n}\int_{\mathbb{R}^n}\sum_{i=1}^k\sign_0^+(w)D^{\beta_i}f\qty\big[\phi_i(u,\tau_{\beta_i}u)-\phi_i(v,\tau_{\beta_i}v)]\omega_i\dd{h}\dd{x}.
		\end{aligned}
	\end{equation}
	Using $\eqref{eq:uniq_Laplace}$ and $\eqref{eq:uniq_B}$ in $\eqref{eq:uniq_main}$ and letting $l\rightarrow\infty$ gives
	\begin{equation}
		\begin{aligned}
			&\int_{\mathbb{R}^n}w^+f\dd{x}-\lambda\int_{\mathbb{R}^n}\int_{\mathbb{R}^n}\sum_{i=1}^k\sign_0^+(w)D^{\beta_i}f\qty\big[\phi_i(u,\tau_{\beta_i}u)-\phi_i(v,\tau_{\beta_i}v)]\omega_i\dd{h}\dd{x}\\
			&-\epsilon\int_{\mathbb{R}^n}w^+\Delta f\dd{x}\le\norm{(g_1-g_2)^+}_{\mathcal{L}^1(\mathbb{R}^n)}.
		\end{aligned}
	\end{equation}
	By the same argument as before, let $\kappa\in\mathcal{C}_0^{\infty}(\mathbb{R})$ nonnegative such that $\kappa(s)=1$ for $|s|\le1$. Set $f_l(\xi)=\kappa\qty\Big(\frac{\norm{\xi}_{\mathbb{R}^n}}{l})$. Since the difference quotient $D^{\beta_i}f_l$ is bounded and is zero for $x\in\mathbb{R}^n$ such that $\norm{x\pm\beta_i}_{\mathbb{R}^n}\le l$  (see \eqref{eq:Dbif}), letting $l\rightarrow\infty$ yields
	\begin{equation}
		\int_{\mathbb{R}^n}w^+\dd{x}=\norm{(u-v)^+}_{\mathcal{L}^1(\mathbb{R}^n)}\le\norm{(g_1-g_2)^+}_{\mathcal{L}^1(\mathbb{R}^n)}.
	\end{equation}
\end{proof}

\begin{cor}\label{cor:regul_L1}
	Let the assumptions of Lemma \ref{lem:u_Lp} hold and let $u,v\in\mathcal{H}^2(\mathbb{R}^n)$ satisfy
	\begin{align}
		&u+Bu-\epsilon\Delta u=g_1\\
		&v+Bv-\epsilon\Delta v=g_2.
	\end{align}
	If $g_1,g_2\in\mathcal{L}^1(\mathbb{R}^n)$, then
	\begin{equation}
		\norm{u-v}_{\mathcal{L}^1(\mathbb{R}^n)}\le\norm{g_1-g_2}_{\mathcal{L}^1(\mathbb{R}^n)}.
	\end{equation}
\end{cor}
\begin{proof}
	Notice that the equality
	\begin{equation}
		\norm{a-b}_{\mathcal{L}^1(\mathbb{R}^n)}=\norm{(a-b)^+}_{\mathcal{L}^1(\mathbb{R}^n)}+\norm{(b-a)^+}_{\mathcal{L}^1(\mathbb{R}^n)}
	\end{equation}
	holds for any $a,b\in\mathcal{L}^1(\mathbb{R}^n)$. Lemma \ref{lem:regul_L1+} shows that
	\begin{align}
		&\norm{(u-v)^+}_{\mathcal{L}^1(\mathbb{R}^n)}\le\norm{(g_1-g_2)^+}_{\mathcal{L}^1(\mathbb{R}^n)},\\
		&\norm{(v-u)^+}_{\mathcal{L}^1(\mathbb{R}^n)}\le\norm{(g_2-g_1)^+}_{\mathcal{L}^1(\mathbb{R}^n)}.
	\end{align}
	Hence, the inequality $\norm{u-v}_{\mathcal{L}^1(\mathbb{R}^n)}\le\norm{g_1-g_2}_{\mathcal{L}^1(\mathbb{R}^n)}$ holds as claimed.
\end{proof}

The next result shows the existence of a unique generalized solution of \eqref{eq:regul} for $g\in\mathcal{L}^1(\mathbb{R}^n)\cap\mathcal{L}^{\infty}(\mathbb{R}^n)$ and plays an essential role in our developments. In order to do so we consider the problem on the ball $B_r\subset\mathbb{R}^n$ for $r>0$ with zero Dirichlet boundary condition. Let $u^r\in\mathcal{H}_0^1(B_r)\cap\mathcal{H}^2(B_r)=:\mathcal{H}_0^2(B_r)$ satisfy
\begin{equation}\label{eq:Dirichlet}
	\begin{aligned}
		&u^r(x)+\lambda Bu^r(x)-\epsilon\Delta u^r(x)=g(x),\qquad&&x\in B_r;\\
		&u^r(x)=0,\qquad&&x\in\partial B_r,
	\end{aligned}
\end{equation}
where $\Delta$ denotes the Dirichlet-Laplacian $\Delta_D$ on $\mathcal{L}^2(B_r)$ with $D(\Delta_D)=\mathcal{H}_0^2(B_r)$. For the operator $B$ to remain meaningful we use the $E:\mathcal{H}_0^{1}(B_r)\mapsto\mathcal{H}^1(\mathbb{R}^n)$ extension operator \cite[Chapter 5.4]{Evans2010} on $u^r$ supplemented with the fact that $\supp(Eu^r)=\supp(u^r)$ and $\norm{Eu^r}_{\mathcal{H}^1(\mathbb{R}^n)}=\norm{u^r}_{\mathcal{H}_0^1(B_r)}$ \cite{Calderon1961}. Then we use the restriction operator $R:\mathcal{L}^2(\mathbb{R}^n)\mapsto\mathcal{L}^2(B_r)$ on $BEu^r$ to obtain the operator $RBE:\mathcal{H}_0^1(B_r)\mapsto\mathcal{L}^2(B_r)$. As in \eqref{eq:Dirichlet}, we will denote $\Delta_D$ by $\Delta$ and $RBE$ by $B$ for brevity.

\begin{rem}\label{rem:dir}
	One can verify from the proof of Lemmata \ref{lem:DB}, \ref{lem:Bcont}, \ref{lem:u_Lp} and \ref{lem:regul_L1+} and Corollaries \ref{cor:u_Lp} and \ref{cor:regul_L1} that they all hold for the Dirichlet problem too. Minor steps of the proofs have to be modified, for example, in the proof of Lemma \ref{lem:u_Lp}, instead of multiplying by $\Phi_l'(u)$ and integrating over $\mathbb{R}^n$ we multiply by $\Phi_l'(Eu^r)$ and integrate over $B_r$. Then we can repeat the same estimates as before. Similar arguments should be used in the rest of the proofs as well.
\end{rem}

\begin{prop}\label{prop:regul_u_exist}
	Let the assumptions of Lemma \ref{lem:u_Lp} hold. Then for each $g\in\mathcal{L}^1(\mathbb{R}^n)\cap\mathcal{L}^{\infty}(\mathbb{R}^n)$ there is a unique solution $u\in\mathcal{H}^2(\mathbb{R}^n)$ of \eqref{eq:regul}.
\end{prop}
\begin{proof}
	We consider the Dirichlet problem \eqref{eq:Dirichlet} first.	Define the operator $T:\mathcal{H}_0^1(B_r)\mapsto\mathcal{H}_0^2(B_r)$ by $T=-(I-\epsilon\Delta)^{-1}\lambda Bu+(I-\epsilon\Delta)^{-1}g$ and let
	\begin{equation}
		\mathcal{S}:=\qty\big{u\in\mathcal{H}_0^1(B_r):u=\eta Tu,~\eta\in[0,1]}.\label{eq:S}
	\end{equation}
	Note that $\mathcal{H}_0^2(B_r)$ can be compactly embedded into $\mathcal{H}_0^1(B_r)$, which implies that $T$ is continuous and compact and maps the Banach space $\mathcal{H}_0^1(B_r)$ into itself. Observe that $u\in\mathcal{S}$ implies in fact $u\in\mathcal{H}_0^2(B_r)$, and thus $u=\eta Tu$ is equivalent to
	\begin{equation}\label{eq:eta}
		u+\eta\lambda Bu-\epsilon\Delta u=\eta g
	\end{equation}
	on $B_r$ a.e. Multiply by $u$ and integrate over $B_r$ to find that
	\begin{equation}
		\begin{aligned}
			\norm{u}_{\mathcal{L}^2(B_r)}^2+\epsilon\norm{\nabla u}_{\mathcal{L}^2(B_r)}^2&=\eta\int_{B_r}gu\dd{x}-\eta\lambda\int_{B_r}Buu\dd{x}\\
			&\le\eta\norm{g}_{\mathcal{L}^2(B_r)}\norm{u}_{\mathcal{L}^2(B_r)}+\eta\lambda\norm{Bu}_{\mathcal{L}^2(B_r)}\norm{u}_{\mathcal{L}^2(B_r)}\\
			&\le\frac{\eta}{2}\norm{g}_{\mathcal{L}^2(B_r)}^2+\frac{\eta}{2}\norm{u}_{\mathcal{L}^2(B_r)}^2+\eta\lambda\delta^2\norm{Bu}_{\mathcal{L}^2(B_r)}^2+\frac{\eta\lambda}{\delta^2}\norm{u}_{\mathcal{L}^2(B_r)}^2\\
			&\le\frac{1}{2}\norm{g}_{\mathcal{L}^2(B_r)}^2+\frac{1}{2}\norm{u}_{\mathcal{L}^2(B_r)}^2+\lambda\delta^2\norm{Bu}_{\mathcal{L}^2(B_r)}^2+\frac{\lambda}{\delta^2}\norm{u}_{\mathcal{L}^2(B_r)}^2
		\end{aligned}
	\end{equation}
	for any $\delta>0$. Using \eqref{eq:Bu_upper} and Corollary \ref{cor:u_Lp} (note that the right-hand side is $\eta g$ in \eqref{eq:eta} and $g$ in \eqref{eq:regul}) we find that
	\begin{equation}\label{eq:Sbound1}
		\norm{u}_{\mathcal{L}^2(B_r)}^2\le\eta\norm{g}_{\mathcal{L}^1(B_r)}\norm{g}_{\mathcal{L}^{\infty}(B_r)}\le\norm{g}_{\mathcal{L}^1(B_r)}\norm{g}_{\mathcal{L}^{\infty}(B_r)}
	\end{equation}
	and that
	\begin{equation}\label{eq:Sbound2}
		\begin{aligned}
			(\epsilon-C\lambda\delta^2)\norm{\nabla u}_{\mathcal{L}^2(B_r)}^2&\le\frac{1}{2}\norm{g}_{\mathcal{L}^2(B_r)}^2+\qty\bigg(\frac{1}{2}+\frac{\lambda}{\delta^2})\norm{u}_{\mathcal{L}^2(B_r)}^2\le\qty\bigg(1+\frac{\lambda}{\delta^2})\norm{g}_{\mathcal{L}^1(B_r)}\norm{g}_{\mathcal{L}^{\infty}(B_r)}.
		\end{aligned}
	\end{equation}
	The inequalities \eqref{eq:Sbound1} and \eqref{eq:Sbound2} show that by choosing $\delta$ small enough $\mathcal{S}$ is bounded in $\mathcal{H}_0^1(B_r)$. Then Schaefer's fixed point theorem shows that $T$ has a fixed point \cite[Corollary 8.1]{Deimling1985} and, in fact, Lemma \ref{lem:regul_L1+} ensures that the fixed point is unique on $B_r$.

	Again, by \cite[Chapter 5.4]{Evans2010} and \cite{Calderon1961} there exists an $E:\mathcal{H}_0^2(B_r)\mapsto\mathcal{H}^2(\mathbb{R}^n)$ extension operator such that $\supp(Eu^r)=\supp(u^r)$ and	$\norm{Eu^r}_{\mathcal{H}^2(\mathbb{R}^n)}=\norm{u^r}_{\mathcal{H}_0^2(B_r)}$. {\color{black}Choose a sequence $\qty{r_m}\subset\mathbb{R}$ such that $r_m\rightarrow\infty$ as $m\rightarrow\infty$ and observe that by the uniqueness of $Eu^{r_m}$ on $B_{r_m}$ for each compact $\Omega\subset\mathbb{R}^n$ there exists $m_0$ such that $Eu^{r_m}=Eu^{r_{m_0}}$ a.e. on $\Omega$ for $m>m_0$. In particular, this implies that $\norm{u^{r_m}}_{\mathcal{H}_0^2(B_{r_{m_0}})}\le C=C(r_{m_0})$. Furthermore, by Lemma \ref{lem:u_Lp} we also have $\norm{Eu^{r_m}}_{\mathcal{L}^{\infty}(B_{r_m})}\le\norm{g}_{\mathcal{L}^{\infty}(B_{r_m})}$. This shows that the sequence $Eu^{r_m}$ converges to some $u\in\mathcal{L}^{\infty}(\mathbb{R}^n)\cap\mathcal{H}_{loc}^2(\mathbb{R}^n)$.} Elliptic regularity \cite[Section 6.3.1]{Evans2010} combined with inequalities \eqref{eq:Sbound1} and \eqref{eq:Sbound2} imply that

	\begin{equation}
		\begin{aligned}
			\norm{u^{r_m}}_{\mathcal{H}_0^2(B_{r_m})}&\le C\qty\big(\norm{g}_{\mathcal{L}^2(B_{r_m})}+\norm{Bu^{r_m}}_{\mathcal{L}^2(B_{r_m})})\le C\qty\big(\norm{g}_{\mathcal{L}^2(B_{r_m})}+\norm{u^{r_m}}_{\mathcal{H}_0^1(B_{r_m})})\\
			&\le C\qty\Big(\norm{g}_{\mathcal{L}^2(\mathbb{R}^n)}+\norm{g}_{\mathcal{L}^1(\mathbb{R}^n)}^{\frac{1}{2}}\norm{g}_{\mathcal{L}^{\infty}(\mathbb{R}^n)}^{\frac{1}{2}}).
		\end{aligned}
	\end{equation}
	Since the sequence $\qty\Big{\norm{u^{r_m}}_{\mathcal{H}_0^2(B_{r_m})}}$ is monotone and bounded it converges to its finite supremum; that is,
	\begin{equation}
		\lim_{r_m\rightarrow\infty}\norm{u^{r_m}}_{\mathcal{H}_0^2(B_{r_m})}=\norm{u}_{\mathcal{H}^2(\mathbb{R}^n)},
	\end{equation}
	which means that $u\in\mathcal{H}^2(\mathbb{R}^n)$ and the proof is complete.
\end{proof}

In our next result we take the limit $\epsilon\rightarrow0$. This will not only allow us to consider flux functions in $\mathcal{W}_{loc}^{1,\infty}(\mathbb{R}\times\mathbb{R})$ but will show that the various properties established for the solutions of \eqref{eq:regul} hold for the generalized solutions of \eqref{eq:very_regul}, which in turn will imply that they hold for the semigroup as well.

\begin{prop}\label{prop:Tl}
	Let $\phi_i\in\mathcal{W}_{loc}^{1,\infty}(\mathbb{R}\times\mathbb{R})$ and $A_0$ be given by Definition \ref{def:A0}. Then $\mathcal{L}^1(\mathbb{R}^n)\cap\mathcal{L}^{\infty}(\mathbb{R}^n)\subseteq R(I+\lambda A_0)$ for $\lambda>0$. Accordingly, let $T_{\lambda}:\mathcal{L}^1(\mathbb{R}^n)\cap\mathcal{L}^{\infty}(\mathbb{R}^n)\mapsto\mathcal{L}^1(\mathbb{R}^n)$ be the restriction of $(I+\lambda A_0)^{-1}$ to $\mathcal{L}^1(\mathbb{R}^n)\cap\mathcal{L}^{\infty}(\mathbb{R}^n)$. If $g_1,g_2\in\mathcal{L}^1(\mathbb{R}^n)\cap\mathcal{L}^{\infty}(\mathbb{R}^n)$, then
	\begin{enumerate}[label=\normalfont(\roman*)]
		\item $T_{\lambda}g_1\in\mathcal{L}^p(\mathbb{R}^n)$ for $p\ge1$ with $\norm{T_{\lambda}g_1}_{\mathcal{L}^p(\mathbb{R}^n)}\le\norm{g_1}_{\mathcal{L}^1(\mathbb{R}^n)}^{\frac{1}{p}}\le\norm{g_1}_{\mathcal{L}^{\infty}(\mathbb{R}^n)}^{1-\frac{1}{p}}$,
		\item $-\norm{g_1^-}_{\mathcal{L}^{\infty}(\mathbb{R}^n)}\le T_{\lambda}g_1\le\norm{g_1^+}_{\mathcal{L}^{\infty}(\mathbb{R}^n)}$,
		\item $\norm{(T_{\lambda}g_1-T_{\lambda}g_2)^+}_{\mathcal{L}^1(\mathbb{R}^n)}\le\norm{(g_1-g_2)^+}_{\mathcal{L}^1(\mathbb{R}^n)}$,
		\item $T_{\lambda}$ commutes with translations,
		\item $\int_{\mathbb{R}^n}T_{\lambda}g_1\dd{x}=\int_{\mathbb{R}^n}g_1\dd{x}$.
	\end{enumerate}
\end{prop}
\begin{proof}
	Let $\qty{\phi_i^m}\subset\mathcal{C}^1(\mathbb{R}\times\mathbb{R})$ be a sequence such that each $\phi_i^m$ is bounded and have the property $\phi_i^m(0,0)=0$ and $\qty{\phi_i^m}$ converges to $\phi_i$ uniformly on compact sets.	Define
	\begin{equation}
		B_mu=\int_{\mathbb{R}^n}\sum_{i=1}^k\frac{\phi_i^m(u,\tau_{\beta_i}u)-\phi_i^m(\tau_{-\beta_i}u,u)}{\norm{\beta_i}_{\mathbb{R}^n}}\omega_i\dd{h}
	\end{equation}
	and $T_{\lambda,m}:\mathcal{L}^1(\mathbb{R}^n)\cap\mathcal{L}^{\infty}(\mathbb{R}^n)\mapsto\mathcal{L}^1(\mathbb{R}^n)\cap\mathcal{L}^{\infty}(\mathbb{R}^n)$ by $T_{\lambda,m}g=u$ if $u\in\mathcal{H}^2(\mathbb{R}^n)$
	and
	\begin{equation}\label{eq:ul}
		u+\lambda B_mu-\frac{1}{m}\Delta u=g.
	\end{equation}
	Proposition \ref{prop:regul_u_exist}, Lemmata \ref{lem:u_Lp} and \ref{lem:regul_L1+}, Remark \ref{rem:infty}, Corollaries \ref{cor:u_Lp} and \ref{cor:regul_L1} and the fact that $T_{\lambda,m}$ commutes with translations imply that $T_{\lambda,m}$ is well-defined and has the properties (i)-(iv). Let $g\in\mathcal{L}^1(\mathbb{R}^n)\cap\mathcal{L}^{\infty}(\mathbb{R}^n)$ and $u_m=T_{\lambda,m}g$. By Lemma \ref{lem:regul_L1+} and the translation invariance of $T_{\lambda,m}$ we conclude that
	\begin{equation}
		\int_{\mathbb{R}^n}\qty\big|u_m(x+y)-u_m(x)|\dd{x}\le\int_{\mathbb{R}^n}\qty\big|g(x+y)-g(x)|\dd{x}
	\end{equation}
	for $y\in\mathbb{R}^n$. The above estimate and $\norm{u_m}_{\mathcal{L}^1(\mathbb{R}^n)}\le\norm{g}_{\mathcal{L}^1(\mathbb{R}^n)}$, by the means of the Fréchet-Kolmogorov compactness theorem, imply that $\qty{u_m}$ is precompact in $\mathcal{L}_{loc}^1(\mathbb{R}^n)$. Thus, there is a subsequence $\qty{u_{m_j}}$ which converges a.e. in $\mathcal{L}_{loc}^1(\mathbb{R}^n)$ to a limit $u\in\mathcal{L}^1(\mathbb{R}^n)$. This convergence will be denoted as $u_{m_j}\twoheadrightarrow u$. Let $f\in\mathcal{C}_0^{\infty}(\mathbb{R}^n)$ be nonnegative and $\Phi_l$ be given by \eqref{eq:Phi}. Multiply \eqref{eq:ul} by $\Phi_l'(u_m-c)f$ and integrate over $\mathbb{R}^n$ to find that
	{\color{black}
	\begin{equation}
		\int_{\mathbb{R}^n}\qty\bigg(u_m+\lambda B_mu_m-\frac{1}{m}\Delta u_m)\Phi_l'(u_m-c)f\dd{x}=\int_{\mathbb{R}^n}g\Phi_l'(u_m-c)f\dd{x}.
	\end{equation}
	Integration by parts gives
	\begin{equation}
		\int_{\mathbb{R}^n}\qty\bigg((u_m-g)\Phi_l'(u_m-c)f+\lambda B_mu_m\Phi_l'(u_m-c)f+\frac{1}{m}\qty\big(\Phi_l''(u_m-c)|\nabla u_m|^2f-\Phi_l(u_m-c)\Delta f))\dd{x}=0.
	\end{equation}
	Note that both $\Phi_l'',f\ge0$ implies that
	\begin{equation}
		\frac{1}{m}\int_{\mathbb{R}^n}\Phi_l''(u_m-c)|\nabla u_m|^2f\dd{x}\ge0
	\end{equation}
	and $\norm{u_m}_{\mathcal{L}^{\infty}(\mathbb{R}^n)}\le\norm{g}_{\mathcal{L}^{\infty}(\mathbb{R}^n)}$ implies that the integral
	\begin{equation}
		\int_{\mathbb{R}^n}\Phi_l(u_m-c)\Delta f\dd{x}
	\end{equation}
	is bounded. Letting $m\rightarrow\infty$ through the subsequence $\qty{m_j}$ and using the convergences $u_{m_j}\twoheadrightarrow u$ and $\phi_i^m\rightarrow\phi_i$ uniformly on compact sets yields
	\begin{equation}
		\int_{\mathbb{R}^n}\qty\big((u-g)\Phi'_l(u-c)f+\lambda Bu\Phi_l(u-c)f)\dd{x}\le0.
	\end{equation}
	}
	Letting $l\rightarrow\infty$ and using \eqref{eq:Bv_main_ineq} gives
	\begin{equation}
		\int_{\mathbb{R}^n}\qty\Bigg(\sign_0(u-c)(u-g)f-\lambda\int_{\mathbb{R}^n}\sum_{i=1}^kD^{\beta_i}\qty\big[f\sign_0(u-c)]\qty\big(\phi_i(u,\tau_{\beta_i}u)-\phi_i(c,c))\omega_i\dd{h})\dd{x}\le0.
	\end{equation}
	Since $\norm{u}_{\mathcal{L}^{\infty}(\mathbb{R}^n)}\le\norm{g}_{\mathcal{L}^{\infty}(\mathbb{R}^n)}$ and $\phi_i\in\mathcal{W}_{loc}^{1,\infty}(\mathbb{R}\times\mathbb{R})$ we have $\phi_i(u,\tau_{\beta_i}u)\in\mathcal{L}^1(\mathbb{R}^n)$. Thus, we have $g\in(I+\lambda A_0)u$ by Definition \ref{def:A0} and, in fact, by Lemma \ref{lem:A0_single} the equality
	\begin{equation}
		u+\lambda A_0u=g\label{eq:eps_limit}
	\end{equation}
	holds. The accretivity of $A_0$ shows that $u$ is unique, hence $\lim_{m\rightarrow\infty}T_{\lambda,m}g=T_{\lambda}g$ holds with convergence in $\mathcal{L}_{loc}^1(\mathbb{R}^n)$. Properties (i)-(iv) are preserved under $\mathcal{L}_{loc}^1(\mathbb{R}^n)$ convergence. Choose $f\in\mathcal{C}_0^{\infty}(\mathbb{R}^n)$ nonnegative, multiply \eqref{eq:eps_limit} with $f$ and integrate over $\mathbb{R}^n$ to find that
	\begin{equation}
		\int_{\mathbb{R}^n}uf\dd{x}+\lambda\int_{\mathbb{R}^n}A_0uf\dd{x}=\int_{\mathbb{R}^n}uf\dd{x}-\lambda\int_{\mathbb{R}^n}\int_{\mathbb{R}^n}\sum_{i=1}^kD^{\beta_i}f\phi_i(u,\tau_{\beta_i}u)\omega_i\dd{h}\dd{x}=\int_{\mathbb{R}^n}gf\dd{x}
	\end{equation}
	also holds by Lemma \ref{lem:A0_single}. Let $\kappa\in\mathcal{C}_0^{\infty}(\mathbb{R})$ be nonnegative such that $\kappa(s)=1$ for $|s|\le1$. Set $f_l(\xi)=\kappa\qty\Big(\frac{\norm{\xi}_{\mathbb{R}^n}}{l})$ and let $l\rightarrow\infty$. Using \eqref{eq:Dbif} we find that the integral
	\begin{equation}
		\int_{\mathbb{R}^n}\int_{\mathbb{R}^n}\sum_{i=1}^kD^{\beta_i}f_l\phi_i(u,\tau_{\beta_i}u)\omega_i\dd{h}\dd{x}
	\end{equation}
	converges to zero as $l\rightarrow\infty$ and thus property (v) holds as well.
\end{proof}

\begin{rem}
	By Definition \ref{def:A0} it is clear that $\overline{D(A)}\subset\mathcal{L}^1(\mathbb{R}^n)$ and in some cases, in fact, the equality $\overline{D(A)}=\mathcal{L}^1(\mathbb{R}^n)$ holds, see Lemma \ref{lem:B_in_A0}. However, this remains to be shown under our general assumption that $\phi_i\in\mathcal{W}_{loc}^{1,\infty}(\mathbb{R}\times\mathbb{R})$.
\end{rem}

\begin{proof}[Proof of Theorem \ref{thm:A_generates}]
	Since $A_0$ is accretive it follows that the closure $A$ is also accretive. Let $g\in\mathcal{L}^1(\mathbb{R}^n)$ and $\qty{g_m}\subset\mathcal{L}^1(\mathbb{R}^n)\cap\mathcal{L}^{\infty}(\mathbb{R}^n)$ be such that $g_m\rightarrow g$ in $\mathcal{L}^1(\mathbb{R}^n)$. Since $T_{\lambda}$ is a contraction, the sequence $\qty{T_{\lambda}g_m}$ is Cauchy. Let $\lambda w_m=(I-T_{\lambda})g_m$, so $w_m\in A_0T_{\lambda}g_m$ and the sequence $\qty{w_m}$ is also Cauchy. If $T_{\lambda}g_m\rightarrow v$ and $w_m\rightarrow w$, then $w\in Av$ and $g=v+\lambda w\in(I+\lambda A)v$. This shows that $A$ is $m$-accretive and the proof is complete.
\end{proof}

\begin{proof}[Proof of Theorem \ref{thm:A_go_brr}]
	The solution $u_{\epsilon}(t)$ of \eqref{eq:genCP} is given by
	\begin{equation}
		u_{\epsilon}(t)=(I+\epsilon A)^{-\big\lfloor\frac{t}{\epsilon}\big\rfloor-1}u_0.
	\end{equation}
	The uniform convergence $\lim_{\epsilon\rightarrow0}u_{\epsilon}(t)=S(t)u_0$ for $t$ in $\mathcal{L}^1(\mathbb{R}^n)$ shows that properties (i)-(v) hold for $S(t)$, since by Proposition \ref{prop:Tl} they hold for $T_{\lambda}=(I+\lambda A)^{-1}$.

	For property (vi) let $u_0\in\mathcal{L}^1(\mathbb{R}^n)\cap\mathcal{L}^{\infty}(\mathbb{R}^n)$ (note that by Lemma \ref{lem:A0_single} the operator $A_0$ is single-valued in this case) and $u_{\epsilon}(x,t)$ satisfy
	\begin{equation}
		\begin{aligned}
			&\frac{1}{\epsilon}\qty\big(u_{\epsilon}(x,t)-u_{\epsilon}(x,t-\epsilon))+A_0u_{\epsilon}(x,t)=0,\qquad&&(x,t)\in\mathbb{R}^n\times(0,T);\\
			&u_{\epsilon}(x,0)=u_0(x),&&x\in\mathbb{R}^n.
		\end{aligned}
	\end{equation}
	The definition of $A_0$ implies that
	\begin{equation}
		\int_{\mathbb{R}^n}\sign_0\qty\big(u_{\epsilon}(x,t)-c)A_0u_{\epsilon}(x,t)f\dd{x}+\int_{\mathbb{R}^n}\int_{\mathbb{R}^n}\sum_{i=1}^kD^{\beta_i}\qty\big[f\sign_0(u-c)]\qty\big(\phi_i(u_{\epsilon},\tau_{\beta_i}u_{\epsilon})-\phi_i(c,c))\omega_i\dd{h}\dd{x}\ge0
	\end{equation}
	holds for any nonnegative $f\in\mathcal{C}_0^{\infty}\qty\big(\mathbb{R}^n\times(0,T))$ and any $c\in\mathbb{R}$. Notice that
	\begin{equation}
		A_0u_{\epsilon}(x,t)=\frac{1}{\epsilon}\qty\big(u_{\epsilon}(x,t-\epsilon)-u_{\epsilon}(x,t))
	\end{equation}
	and that
	\begin{equation}
		\begin{aligned}
			&\sign_0\qty\big(u_{\epsilon}(x,t)-c)\qty\big(u_{\epsilon}(x,t-\epsilon)-u_{\epsilon}(x,t))=\sign_0\qty\big(u_{\epsilon}(x,t)-c)\qty\big(u_{\epsilon}(x,t-\epsilon)-c)\\
			&+\sign_0\qty\big(u_{\epsilon}(x,t)-c)\qty\big(u_{\epsilon}(x,t)-c)\le\qty\big|u_{\epsilon}(x,t-\epsilon)-c|-\qty\big|u_{\epsilon}(x,t)-c|.
		\end{aligned}
	\end{equation}
	Using the above and integrating over $(0,T)$ yields
	{\color{black}
	\begin{equation}
		\begin{aligned}
			&\int_0^T\int_{\mathbb{R}^n}\frac{1}{\epsilon}\qty\Big(\qty\big|u_{\epsilon}(x,t-\epsilon)-c|-\qty\big|u_{\epsilon}(x,t)-c|)f(x,t)\dd{x}\dd{t}\label{eq:epsilon}\\
			&+\int_0^T\int_{\mathbb{R}^n}\int_{\mathbb{R}^n}\sum_{i=1}^kD^{\beta_i}\qty\big[f\sign_0(u-c)]\qty\big(\phi_i(u_{\epsilon},\tau_{\beta_i}u_{\epsilon})-\phi_i(c,c))\omega_i\dd{h}\dd{x}\dd{t}\ge0.
		\end{aligned}
	\end{equation}
	}
	Observe that
	\begin{equation}
		\begin{aligned}
			&\frac{1}{\epsilon}\int_0^T\int_{\mathbb{R}^n}\qty\Big(\qty\big|u_{\epsilon}(x,t-\epsilon)-c|-\qty\big|u_{\epsilon}(x,t)-c|)f(x,t)\dd{x}\dd{t}\\
			&=\frac{1}{\epsilon}\qty\Bigg(\int_0^{\epsilon}\int_{\mathbb{R}^n}\qty\big|u_{\epsilon}(x,t-\epsilon)-c|f(x,t)\dd{x}\dd{t}-\int_{T-\epsilon}^T\int_{\mathbb{R}^n}\qty\big|u_{\epsilon}(x,t)-c|f(x,t)\dd{x}\dd{t})\\
			&+\int_{\epsilon}^{T-\epsilon}\int_{\mathbb{R}^n}\qty\big|u_{\epsilon}(x,t)-c|\frac{1}{\epsilon}\qty\big(f(x,t+\epsilon)-f(x,t))\dd{x}\dd{t}.
		\end{aligned}
	\end{equation}
	Since $f\in\mathcal{C}_0^{\infty}\qty\big(\mathbb{R}^n\times(0,T))$ the first two integrals after the equal sign vanish for $\epsilon$ small enough. The uniform convergence $\lim_{\epsilon\rightarrow0}u_{\epsilon}(x,t)=S(t)u_0(x)$ in $\mathcal{L}^1(\mathbb{R}^n)$ implies that the third integral tends to
	\begin{equation}
		\int_0^T\int_{\mathbb{R}^n}\qty\big|S(t)u_0(x)-c|\pdv{f}{t}\dd{x}\dd{t};
	\end{equation}
	that is, by taking the limit $\epsilon\rightarrow0$ in $\eqref{eq:epsilon}$ the proof is complete.
\end{proof}

\bibliographystyle{abbrv}

\end{document}